\RequirePackage{rotating}
\documentclass[letterpaper,12pt]{article}
\usepackage{latexsym,amssymb,amsfonts,amsmath,nccmath,enumitem,setspace,graphicx,float}
\usepackage[dvipsnames]{xcolor}
\usepackage[margin=3cm]{geometry}
\usepackage{dsfont}
\usepackage{mathrsfs}
\newcommand{\cmmnt}[1]{}
\usepackage{amsthm}
\usepackage{comment}
\usepackage[table]{xcolor}
\usepackage{float}
\usepackage{multicol}
\usepackage[normalem]{ulem}
\usepackage{tikz}
\usepackage{ytableau}
\usepackage{cancel}
\usetikzlibrary{graphs,graphs.standard}
\tikzstyle{vertex}=[circle, draw=black, fill=black, minimum size=2pt, inner sep=2]
\usepackage{subcaption}
\usepackage{color}
\usepackage{pdflscape}
\usepackage{pdflscape}

\theoremstyle{plain}
\newtheorem{theorem}{Theorem}[section]
\newtheorem{lemma}[theorem]{Lemma}
\newtheorem{coro}[theorem]{Corollary}
\newtheorem{prop}[theorem]{Proposition}

\newtheorem{defi}[theorem]{Definition}
\theoremstyle{definition}
\newtheorem{conjecture}[theorem]{Conjecture}

\DeclareMathOperator{\dist}{dist}

\usepackage{hyperref}

\usepackage[maxfloats=256]{morefloats}

\def \deg {{\rm deg}}

\def \leq {\leqslant}
\def \geq {\geqslant}

\begin{document}

\title{On the second largest eigenvalue of certain graphs in the perfect matching association scheme}

\author{Himanshu Gupta \footnotemark[1], Allen Herman \footnotemark[1], Alice Lacaze-Masmonteil \footnotemark[1], \\ Roghayeh Maleki \footnotemark[1], Karen  Meagher \footnotemark[1]}

\footnotetext[1]{Department of Mathematics and Statistics, University of Regina, Regina, SK, Canada}


\maketitle

\begin{center}
{\bf Abstract}
\end{center}

The perfect matching association scheme is a set of relations on the perfect matchings of the complete graph on $2n$ vertices. The relations between perfect matchings are defined by the cycle structure of the union of any two perfect matchings, and each relation can be represented as a matrix. Each matrix is labeled by an integer partition whose parts correspond to the size do the cycles in the union. Since these matrices form an association scheme, they are simultaneously diagonalizable. Further, it is well-known that the common eigenspaces correspond to the irreducible representations of $S_{2n}$ indexed by the even partitions of $2n$. In this paper, we conjecture that the second largest eigenvalue of the matrices in the perfect matching association scheme labeled by a partition containing at least two parts of size 1 always occurs on the eigenspace corresponding to the representation indexed by  $[2n-2, 2]$. We confirm this conjecture for matrices labeled by the partitions $[2, 1^{n-2}], [3, 1^{n-3}], [2, 2, 1^{n-4}], [4, 1^{n-4}], [3, 2, 1^{n-5}]$, and $[5, 1^{n-5}]$,  as well as any partition in which the first part is sufficiently large.


\section{Introduction}\label{S:intro}
Let $X$ be a graph with vertex set $V(X)$ and edge set $E(X)$. We use standard terminology from graph theory (see, e.g., \cite[Ch.~1]{GodRoy}). A graph is called {\em $k$-regular} if every vertex has degree $k$. Throughout, all graphs considered are regular and simple (i.e., undirected, without loops or multiple edges). The {\em adjacency matrix} $A$ of a graph $X$ is the matrix indexed by $V(X)$ such that $A_{x,y} = 1$ if $x$ is adjacent to $y$, and $A_{x,y} = 0$ otherwise. For a $k$-regular graph $X$, it is well-known that the largest eigenvalue of $A$ is $k$, and if $X$ is connected, this eigenvalue has multiplicity one. Of particular interest is the second largest eigenvalue of a connected graph $X$, as it carries significant implications for the graph-theoretic properties of $X$. 
 
Namely, the {\em spectral gap} of a graph, defined as the difference between its largest and second largest eigenvalues, plays an important role in understanding connectivity of the graph. For a connected $k$-regular graph, the spectral gap coincides with algebraic connectivity, a concept introduced by Fiedler \cite{fiedler1973algebraic}, which quantifies how well connected the graph is. The spectral gap is central in the study of expander graphs, a class of highly connected sparse graphs with wide-ranging applications in computer science and mathematics. In particular, a larger spectral gap implies stronger expansion properties and faster mixing of random walks on the graph (see, e.g., \cite{alon1985lambda1}, \cite{chung1997spectral}, \cite[Ch.~13]{GodRoy}, \cite{hoory2006expander}, \cite{lovasz1993random}). Therefore, the second largest eigenvalue is important for a graph's algebraic, structural, and probabilistic properties.

In this paper, we aim to identify the second largest eigenvalue of certain graphs in the perfect matching association scheme and compute the corresponding spectral gap. 

\begin{defi} \label{defn:asso}
Let $U$ be a set of $n$ points. A \emph{symmetric association scheme} is a finite set of $n \times n$ binary matrices $\mathcal{A}=\{A_0, A_1, \ldots, A_{d}\}$ such that: 
\begin{enumerate}
\item each $A_i$ is symmetric for all $0\leq i\leq d$;
\item $A_0=I_{n}$, the $n\times n$ identity matrix; 
\item $\sum_{i=0}^dA_i=J_n$, the $n\times n$ all-ones matrix; 
\item $A_iA_j=A_jA_i$ for all $0\leq i,j\leq d$; and
\item for all $0\leq i,j\leq d$, the product $A_iA_j$ lies in the linear span of $\{A_0, A_1, \ldots, A_{d}\}$ over $\mathbb{C}$. 
\end{enumerate}
Two points $u_1, u_2\in U$ are called \emph{$i$-related} if and only if $A_i(u_1, u_2)=1$. The indices $i=0,1,\ldots,d$ are referred to as the \emph{relations} of the scheme. Each matrix $A_i$ is called an \emph{associate} and serves as the adjacency matrix of the corresponding relation on $U$. Each associate $A_i$ is the adjacency matrix for a graph on the vertex set $U$, these graphs are denoted by $X_i$ and are called the graphs of the association scheme
\end{defi}

By axioms (1) and (5) of Definition \ref{defn:asso}, each row and column of associate $A_i$ has the same number of non-zero entries, which we denote as $v_i$; the \emph{valency} of $A_i$. In other words, the graphs $X_i$ of the scheme are $v_i$-regular. See \cite[Ch.~3]{GodMeaBook} for a brief but thorough introduction to association schemes. 

A \textit{perfect matching} of a graph $X$ is a $1$-regular spanning subgraph of $X$. In this paper, we are interested in a particular association scheme known as the \emph{perfect matching association scheme}. This is an association scheme whose set of points are the perfect matchings of $K_{2n}$, the complete graph on $2n$ vertices. In order to define the relations of this scheme, we use integer partitions of $n$. If $\mu$ is a partition of $n$ with $k$ parts $\mu_1\geq \mu_2\geq \cdots\geq \mu_k$ satisfying $\sum_{i=1}^{k}\mu_i=n$, we write $\mu=[\mu_1, \mu_2, \ldots, \mu_k]$ and $\mu \vdash n$. 

\begin{defi} The \emph{perfect matching association scheme} is an association scheme whose set of points is the set of the perfect matchings of $K_{2n}$. This association scheme is the set $\mathcal{A}_{2n}=\{A_\mu \ | \ \mu \vdash n\}$ such that, if $\mu=[\mu_1, \mu_2, \ldots, \mu_k]$ is a partition of $n$, then two perfect matchings are $\mu$-related if and only if their union is a set of vertex-disjoint cycles of lengths $2\mu_1, 2\mu_2, \ldots, 2\mu_k$.  
\end{defi}

See Figure \ref{fig:example} for an illustration of two perfect matchings of $K_{12}$ that are adjacent in $X_{[3,2,1]}$. There is a clear bijection between the even partitions of $2n$ (i.e., each part is even) and the set of all partitions of $n$. Therefore, we have a matrix in $\mathcal{A}_{2n}$ for each even partition of $2n$. In Section \ref{S:term}, we further describe $\mathcal{A}_{2n}$ and explain why this set of matrices forms a symmetric association scheme. 

\begin{figure}[htpb]
\begin{center}
\begin{tikzpicture}[
  very thick,
  every node/.style={circle,draw=black,fill=black!90}, inner sep=2]
 
  \node (n0) at (0.6,2.0) [label=above:$1$] {};
  \node (n1) at (1.7,2.0) [label=above:$2$]{};
  \node (n2) at (0.0,1.0) [label=left:$6$]{};
  \node (n3) at (2.3,1.0) [label=right:$3$]{};
  \node (n4) at (0.6,0.0) [label=below:$5$]{};
  \node (n5) at (1.7,0.0)[label=below:$4$] {};
  \node (n6) at (5,2.0) [label=above:$7$]{};
  \node (n7) at (4,1.0) [label=left:$10$]{};
  \node (n8) at (6,1.0) [label=right:$8$]{};
  \node (n9) at (5,0.0) [label=below:$9$]{};
  \node (n10) at (7.7,1.6) [label=above:$11$]{};
  \node (n11) at (7.7,0.2) [label=below:$12$]{};
 
    \path[every node/.style={font=\sffamily\small}]
   (n11) edge [bend left=30] (n10)
   (n0) edge (n2)
   (n1) edge (n3)
   (n4) edge (n5)
   (n6) edge (n8)
   (n9) edge (n7);

   \path[every node/.style={font=\sffamily\small},  dashed]
   (n11) edge (n10)
   (n0) edge (n1)
   (n2) edge (n4)
   (n3) edge (n5)
   (n6) edge (n7)
   (n8) edge (n9);
\end{tikzpicture}
\end{center}
\caption{Two perfect matchings of $K_{12}$ drawn in dashed and solid, respectively, that overlap to form three cycles of lengths $6, 4$, and $2$, respectively.}
\label{fig:example}
\end{figure}
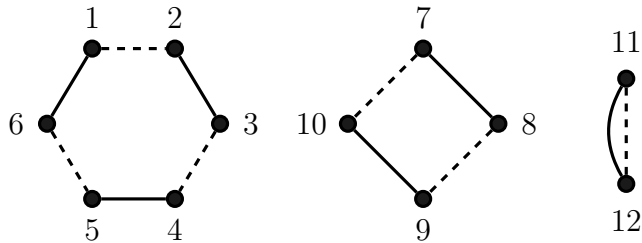

Since $\mathcal{A}_{2n}$ is a symmetric association scheme, the matrices in $\mathcal{A}_{2n}$ pairwise commute by the axiom (4) of Definition \ref{defn:asso}. Consequently, they are simultaneously diagonalizable, and thus, have the same eigenspaces. The eigenspaces of the matrices in $\mathcal{A}_{2n}$ correspond to the irreducible representations of the symmetric group $S_{2n}$ indexed by the even partitions of $2n$ \cite{Decomp}; see \cite[Ch.~2]{Sagan} for a description of these irreducible representations.  Although the eigenspaces are known, much less is understood about the spectrum of the matrices of $\mathcal{A}_{2n}$. The largest eigenvalue of $A_\mu$ is its valency $v_{\mu}$ and will occur on the $[2n]$-eigenspace \cite{Lindzey, MacDonald}. Macdonald \cite{MacDonald} provides explicit formulae to compute the eigenvalues occurring on three eigenspaces and derives complete spectra for two matrices in $\mathcal{A}_{2n}$.The graph $X_{[2,1,\ldots,1]}$, also known as the {\em flip graph on perfect matching of complete graphs}, is well-studied. Diaconis and Holmes \cite{DiaHol} studied random walks on $X_{[2,1,\ldots,1]}$ and computed its eigenvalues. Lasrtly, Lindzey \cite{JLindzey} derived closed-form formulas for the spectrum of $A_\mu$ using Jack polynomials and Srinivasan \cite{Srinivasan} defined an inductive algorithm that can be used to obtain closed-form formulae for the spectrum of $A_\mu$ in terms of certain content evaluating power symmetric functions on $2n$ variables. 

The spectrum of matrices in the perfect matching association scheme has also been studied in the context of Erd\H{o}-Ko-Rado theorems on perfect matchings. Of particular interest to researchers is the perfect matching derangement graph. This graph is constructed relative to the set of partitions $S=\{\mu \ | \ \mu \vdash n, 1\not \in \mu\}$, that is, $S$ contains all partitions of $n$ with no parts of size $1$. The \emph{perfect matching derangement graph} is the graph whose adjacency matrix is $M(2n)=\sum_{\mu \in S} A_\mu$. Each eigenspace of $M(2n)$ is the union of some of the eigenspaces of the perfect matching association scheme meaning that the eigenvalues of $M(2n)$ are sums of eigenvalues of $A_\mu$. The spectrum of $M(2n)$ is instrumental in obtaining bounds on the largest independent set of the perfect matching graph, thereby yielding tight upper-bounds on the size of the largest $t$-intersecting sets of perfect matchings \cite{FalMea, Lindzey, tLindzey}.  Other properties of the spectrum of $M(2n)$ have also been examined. Namely, Koh et al. \cite{AltSign} and Renteln \cite{Renteln} confirmed a conjecture of Meagher and Godsil \cite{GodMeaBook} and Lindzey \cite{Lindzey} by showing that $M(2n)$ has the alternating sign property.

The main focus of this paper is the second largest eigenvalue of certain graphs in $\mathcal{A}_{2n}$. As stated above, this eigenvalue is of particular interest since it can be used to make certain deductions on the connectivity of the graph.  The aim of this paper is two-fold. First, we aim to identify the eigenspace on which the second largest eigenvalue of certain matrices occur. The graphs of interest are those that are indexed by partitions  $\mu$ with at least two parts of size one. Secondly, once we have identified our second largest eigenvalue, we aim to compute the spectral gap of the corresponding graphs. To that end, we make the following conjecture. 
\begin{conjecture}\label{conj:main} 
Let $\mu$ be a partition of $n$ with at least two parts of size $1$, or $\mu=[n-1,1]$ with $n\neq 4$. Then the second largest eigenvalue of $A_{\mu}$ occurs on the irreducible representation of $S_{2n}$ indexed by $2\lambda$ where $\lambda = [n-1,1]$. 
\end{conjecture}
We verified that Conjecture \ref{conj:main} holds for all $n \leqslant 15$. As already mentioned, Srinivasan \cite{Srinivasan} developed a recursive algorithm, implemented in Maple, for computing the eigenvalues of the graphs in the perfect matching association scheme.\footnote{Srinivasan’s original Maple program is available at \url{https://www.math.iitb.ac.in/~mks/papers/EigenMatch.pdf}}. We translated this algorithm into SageMath and used it to compute the eigenvalue tables and verify the conjecture for $n \leqslant 15$.\footnote{The SageMath implementation, accompanying Jupyter notebook, and eigenvalue data (in Excel and PDF formats) are available at \url{https://github.com/Himanshugupta23/Perfect-Matching-Association-Scheme}. The corresponding tables of eigenvalues (up to $n=15$) are also available online in Google Sheets: \url{https://docs.google.com/spreadsheets/d/1wRQWCRfhrKxmFPYhRuKKMp7pMdueVNhP5XqE2PHgHH4/edit?gid=695415463\#gid=695415463}} The tables containing eigenvalues for $n \leq 7$ are also included in the Appendix~\ref{appendix:small_tables} of this paper.
In this paper, we confirm Conjecture \ref{conj:main} for partitions of the form $\mu=[n-k, \mu']$, where $\mu' \vdash k$, when $n$ is large enough with respect to $k$ and for

$$
\mu \in \{[n-1,1],[2,1^{n-2}], [3,1^{n-3}], [2,2,1^{n-4}], [4,1^{n-4}], [3,2, 1^{n-5}], [5,1^{n-5}] \}.
$$

This paper is structured as follows. In Section \ref{S:term}, we provide the necessary terminology and background.  We will confirm Conjecture \ref{conj:main} in one of two ways for the aforementioned partitions $\mu$. The first method, described in Section \ref{S:trace}, makes use of the trace of $A_{\mu}^2$ to show that the second largest eigenvalue in absolute value for partitions of the form $\mu=[n-k, \mu']$, where $\mu' \vdash k$ and $n$ is large enough with respect to $k$, occurs on the irreducible representation of $S_{2n}$ indexed by $[2n-2, 2]$. This statement implies Conjecture \ref{conj:main} in this case since, when $\mu$ contains at least two parts of size 1, the eigenvalue of interest is always positive. In Section \ref{S:sym}, we use the existing combinatorial formulae for the spectrum of $A_{\mu}$ obtained from \cite{Srinivasan}, in conjunction with an inductive argument, to prove Conjecture \ref{conj:main} for partitions $\mu$ given above for all $n \geqslant 15$. Lastly, in Section \ref{S:gap}, we compute the difference between the eigenvalues occurring on the $[2n]$-eigenspace and the  $[2n-2,2]$-eigenspace for certain partitions $\mu$. We conjecture this number to be the spectral gap of the corresponding $X_\mu$. 

\section{Preliminaries}\label{S:term}

Our first objective is to give a more detailed description of the perfect matching association scheme. We begin by giving key definitions. Given a positive integer $n$, a list $\lambda = [\lambda_1, \lambda_2,\ldots, \lambda_k]$ of positive integers is a \emph{partition} of $n$ if $n=\sum_i\lambda_i$. If $\lambda_i$ occurs $m_i$ times  in  $\lambda\vdash n$, then we abbreviate by writing $\lambda_i^{m_i}$ and say that $\lambda_i$ has multiplicity $m_i$. We will adopt the standard convention that $\lambda_i\geqslant \lambda_{i+1}$. A particular type of partition of interest is a \emph{hook partition} which is a partition of $n$ of the form $[\ell, 1^{n-\ell}]$ where $\ell\geqslant 2$. A partition in which all parts are even is called an \emph{even partition}.  If $\lambda=[\lambda_1, \lambda_2, \ldots,\lambda_k]$ is a partition of $n$, then $2\lambda=[2\lambda_1, 2\lambda_2, \ldots,2\lambda_k]$ is an even partition of $2n$. Clearly,  $g:  \lambda \rightarrow 2\lambda$ is a bijection between all partitions of $n$ and all even partitions of $2n$. 

Next, we define a partial ordering on the partitions of $2n$ known as the \emph{dominance ordering}. Let $\lambda=[\lambda_1, \lambda_2, \ldots, \lambda_{\ell}]$ and $\lambda'=[\lambda_1', \lambda_2', \ldots, \lambda_k']$ be two partitions of $n$. We write $\lambda \trianglerighteq \lambda'$ if $\lambda_1+\lambda_2+\cdots+\lambda_i \geqslant \lambda'_1+\lambda_2'+\cdots+\lambda_i'$ for all $i \geqslant 1$. If $i>\ell$ ($i>k$), then we assume that $\lambda_i=0$ ($\lambda'_i=0$). Under this ordering, the set of partitions of $n$ forms a partially ordered set.

Each partition of $n$ can also be visually described using a Young tableau. Let $\lambda=[\lambda_1, \lambda_2, \ldots, \lambda_k]$ be a partition of $n$. The \emph{Young tableau} of $\lambda$ is a left-justified table with $n$ boxes comprised of $k$ rows such that row $i$ contains $\lambda_i$ boxes.  

The eigenspaces of $\mathcal{A}_{2n}$ are intrinsically linked to the irreducible representations of $S_{2n}$.  For a thorough treatment of these irreducible representations, we refer the reader to \cite[Ch.~2]{Sagan}. The irreducible representations of $S_{2n}$ are indexed by the partitions of $2n$. For each $\lambda \vdash 2n$, we have an irreducible representation in the form of a Specht module $S^{\lambda}$. By $\chi^{\lambda}(\sigma)$, we denote the irreducible character evaluated at $\sigma\in S_{2n}$ for the irreducible representation $S^{\lambda}$. 

As stated in the introduction, the perfect matching association scheme is the association scheme whose set of points is the set of perfect matchings of $K_{2n}$. We will let $\mathcal{M}_{2n}$ denote the set of perfect matchings of $K_{2n}$. If $V(K_{2n})=\{1,2,\ldots, 2n\}$ and $P$ is a perfect matching of $K_{2n}$, then we write $P=a_1 a_2 | a_3 a_4 | \cdots | a_{2n-1} a_{2n}$, where $a_i \in \{1,2,\ldots, 2n\}$ and each pair $a_ia_j$ is an edge of $P$. An elementary counting argument shows that there are $(2n-1)(2n-3)(2n-5)\cdots 3=(2n-1)!!$ distinct perfect matchings of $K_{2n}$, thereby implying that all graphs in $\mathcal{A}_{2n}$ have $(2n-1)!!$ vertices. Note that we use the notation $(2n)!!=(2n)(2n-2)(2n-4)\cdots2 = 2^n (n!)$.  Let $\sigma \in S_{2n}$ and $P$ be a perfect matching of $K_{2n}$. By $\sigma P$, we denote the perfect matching formed by permuting the elements in $V(K_{2n})$ via $\sigma$. Under this action of $S_{2n}$ on $\mathcal{M}_{2n}$, we see that $S_{2n}$ acts transitively on the set of perfect matchings $\mathcal{M}_{2n}$. 
 
Next, we give a detailed description of the perfect matching association scheme with the aim of further describing the eigenspaces of the scheme. To do so, we consider the natural action of $S_{2n}$ on the set of perfect matchings $\mathcal{M}_{2n}$. The stabilizer of a single perfect matching is the subgroup $H_n= S_2 \wr S_n$ also known as the {\em hyperoctahedral subgroup}. Observe that $|H_n|=2^n (n!)$ and thus, $[S_{2n} : H_n]=(2n-1)!!$. There is a bijection between the cosets of $H_n$ and the set of perfect matchings of $K_{2n}$ defined relative to the perfect matching $P=1 \, 2 | 3\,4 | \cdots | 2n-1 \, 2n$ and its stabilizer $H_n$. If $\mu \in S_{2n}$, we map $\mu P$ to the coset $\mu H_n$. The group $S_{2n}$ acts on the cosets of $H_n$ by right multiplication. The permutation matrices that arise from the action of $S_{2n}$ on these cosets is the induced representation $1{\uparrow}^{S_{2n}}_{H_n}$. This induced representation admits the following decomposition into irreducible representations of $S_{2n}$ \cite{Decomp}:
 
 $$1{\uparrow}^{S_{2n}}_{H_n}=\oplus_{\lambda \vdash n}S^{2\lambda}.$$
 
Note that $H_n$ is multiplicity-free. The group $S_{2n}$ also acts on pairs of cosets of $H_n$. The orbits of this action are called the \emph{orbitals}. Two cosets belong to the $\mu_1, \ldots, \mu_k$ orbital exactly when the union of their corresponding perfect matchings form a set of cycles with lengths $2\mu_1, \ldots, 2\mu_k$. The set of pairs of cosets in the same orbital corresponds to the edges of a graph. The adjacency matrices of the graphs that correspond to the orbitals of $H_n$ are precisely the associates of $\mathcal{A}_{2n}$.   In addition, since $H_n$ is multiplicity-free, the pair $(S_{2n}, H_n)$ forms a \emph{Gelfand pair}. The orbitals that arise from a Gelfand pair form a Schurian association scheme, thereby implying that $\mathcal{A}_{2n}$ is in fact an association scheme. 

In addition, the matrices in $\mathcal{A}_{2n}$ commute with the permutation matrices of $1{\uparrow}^{S_{2n}}_{H_n}$. Consequently, these permutation matrices are endomorphisms of the eigenspaces of the associates, meaning that the eigenspaces of the associates must be irreducible representations of $S_{2n}$ that appear in the decomposition of  $1{\uparrow}^{S_{2n}}_{H_n}$ given above. In conclusion, the eigenspaces of  $\mathcal{A}_{2n}$ are unions of the irreducible representations of the symmetric group $S_{2n}$ indexed by the even partitions of $2n$. Note that eigenspaces are described as unions of irreducible representations because an associate $A_\mu$ may have the same eigenvalue $\tau$ occur on two or more modules $S^\lambda$. We adopt the standard convention of referring to the Specht modules indexed by the even partitions of $2n$ as the common eigenspaces of our association scheme and describing the eigenspaces of the individual associates as union of these Specht modules. 

Next, we introduce some key notation to describe the eigenvalues of our scheme. Let $\mu$ and $\lambda$ be two partitions of $n$. Throughout this paper, we will use $\mu$ to refer to the relation between two perfect matchings and $2\lambda$ to index the eigenspaces of our association scheme. We shall then write the eigenvalue of $X_\mu$ on the $2\lambda$-eigenspace as $\phi_{\mu}^{\lambda}$. Although $\lambda $ and $\mu$ are partitions of $n$, it is understood that the corresponding eigenspace and relation corresponds to the even partitions $2\lambda$ and $2\mu$, respectively. We say that $\phi_{\mu}^{\lambda}$ is the eigenvalue of $X_\mu$ that occurs on the $\lambda$-eigenspace. 

The \emph{matrix of eigenvalues} of a symmetric association scheme, also known as the \emph{character table of the scheme}, is a table whose rows are indexed by the common eigenspaces of the associates and whose columns are indexed by the relations of the scheme. In the context of the perfect matching association scheme, the rows and columns are indexed by even partitions of $2n$. We adopt the following ordering: rows are in decreasing order from top to bottom, while columns are in increasing order from left to right where the order is in terms of the dominance ordering. See Appendix~\ref{appendix:small_tables} for the character tables of the perfect matching association scheme for $2\leq n \leq 7$.

\subsection{Combinatorial formulae for the spectrum of \texorpdfstring{$\mathcal{A}_{2n}$}{A{2n}}}
We now survey the known combinatorial formulae for the eigenvalues of the certain matrices and eigenspaces of our scheme. There are three eigenspaces for which we have combinatorial formulae for all graphs in the association scheme. These eigenspaces are indexed by the partitions $[n], [1^{n}]$,  and $[n-1,1]$, respectively.  In \cite{MacDonald}, Macdonald derived a combinatorial formula for computing the degree of $A_\mu$, denoted $v_{\mu}$, which always occurs on the $[n]$-eigenspace. However, Macdonald's results are not explicitly stated as such. Lindzey \cite{Lindzey} translated the work of Macdonald to obtain the following result on the valency of the graphs of $\mathcal{A}_{2n}$. 

\begin{lemma}\cite{MacDonald}  \label{lem:degree}
Let  $\mu = [\mu_1^{m_1}, \dots,  \mu_k^{m_k}]$ and $n=\sum_{i=1}^k m_i \mu_i $. Then the valency of $A_\mu$ is given by
$$v_{\mu} = \phi^{[n]}_{\mu} =\frac{2^n n!}{2^{m_1+\dots+m_k} \prod_i (m_i!) (\mu_i^{m_i})}. $$
\end{lemma} 

Of particular importance to our investigation is the eigenvalue of $X_\mu$ that occurs on the $[n-1,1]$-eigenspace. In this case, we have the following combinatorial formula, also found in \cite{MacDonald}.

\begin{lemma}\cite{MacDonald} \label{lem:foreig}
Let  $\mu$ be a partition of $n$ and $r_1(\mu) \geq 0$ denote the multiplicity of $1$ in $\mu$. Then
$$\phi^{[n-1,1]}_{\mu} = v_{\mu} \left(\dfrac{ (2n-1) r_1(\mu) - n}{2n(n-1)}\right).$$
\end{lemma}

Observe that $\phi^{[n-1,1]}_{\mu}$ is negative when $\mu$ does not contain a part of size 1, thereby implying that it cannot be the second largest eigenvalue. This observation explains why our conjecture only pertains to partitions with at least two parts of size 1.  

In \cite[Ch.~13]{GodMeaBook}, Godsil and Meagher describe a formula for the eigenvalues of orbital scheme that arise from a Gelfand pair. In Lemma \ref{lem:formulaeobvious} below we give this formula in terms of the Gelfand pair $(S_{2n}, H_n)$. This formula is derived by constructing an eigenvector $\mathbf{x}$ from the $\lambda$-eigenspace and evaluating $A_\mu \mathbf{x}$.  Note that this formula can also be found in \cite{bannai} and is known as a \emph{zonal spherical function} evaluated over the elements of a particular coset of $H_n$. 

\begin{lemma}\cite[Lemma 13.8.3]{GodMeaBook} \label{lem:formulaeobvious}
Let $H_n=S_2 \wr S_n$ and $x_\mu \in S_{2n}$ such that $(H_n, x_{\mu} H_n)$ is a pair of cosets in the $\mu$-orbital of $H_n$. Then

$$\phi_{\mu}^\lambda=\frac{v_{\mu}}{2^n(n!)}\sum_{h \in H_n} \chi^\lambda(x_\mu h).$$
\end{lemma}

Although the approach given by Lemma \ref{lem:formulaeobvious} appears to be sensible, it is in fact impractical since we must first extract the elements of a single coset $x_\mu H_n$ and then derive the sum of the irreducible $\lambda$-characters. 

Combinatorial formulae for the spectrum of several matrices are also known. In \cite{DiaHol},  Diaconis and Holmes  derived a formula for the spectrum of $A_{[2,1^{n-2}]}$. Most combinatorial formula for the spectrum of matrices in $\mathcal{A}_{2n}$ can be found in \cite{Srinivasan} where Srinivasan derives an inductive algorithm to derive combinatorial formulae for the spectrum of  matrices of the perfect matching association scheme. These formulae are given in terms of content evaluating symmetric functions. Below, we describe Srinivasan's constructions as they are instrumental to our computations in Section \ref{S:sym}.

To describe Srinivasan's approach, we first define a set of symmetric functions evaluated relative to the content of a Young tableau of a partition $2\lambda$. First, we describe the content of Young tableaux. Let $f(x_1, x_2, \ldots, x_{2n})$ be a symmetric function of $2n$ variables. We assign $x_1$ to the top-leftmost box of the Young tableau of $\lambda$. Variables are then assigned from left-to-right and then top-to-bottom to the $2n$ boxes of $2\lambda$. See Figure \ref{subfig:inva} for the assignment of $2n$ variables to the boxes of a Young tableau for $n=6$. In \cite{Srinivasan} and this paper, the \emph{content} of the box $b$ located at the $(i,j)$-coordinate is $j-i$. See Figure \ref{fig:tabcon} for an example of the content of $2\lambda=[6,4,2]$. The $2n$-tuple $c(\lambda)$ will denote the content of the boxes of $2\lambda$.  If $f(x_1, x_2, \ldots, x_{2n})$ is a symmetric function of $2n$ variables, then $f(c(\lambda))$ is $f$ evaluated at the content of the Young tableau of  $2\lambda$. 

\ytableausetup{mathmode, boxsize=1.7em}
\begin{figure} [htpb]
\begin{subfigure}{0.49\textwidth}
\centering
\begin{ytableau}
x_1 & x_2&x_3 &x_4&x_5&x_6\\
x_7 & x_8 & x_9 &x_{10}\\
x_{11} & x_{12} 
\end{ytableau}
\caption{Assignment of $2n$ variables to the boxes of a Young tableau for $n=6$.}
\label{subfig:inva}
\end{subfigure}
\begin{subfigure}{0.49\textwidth}
\centering
\begin{ytableau}
0 & 1&2 &3&4 &5\\
-1 & 0 & 1 &2\\
-2 & -1 
\end{ytableau}
\caption{Content of Young tableau corresponding to the partition $[6,4,2]$.}
\label{fig:tabcon}
\end{subfigure}
\caption{Young tableau for partition $2\lambda=[6,4,2]$ with content.}
\end{figure}

To derive formulae for the spectrum of the relations of our scheme, we will work  in the algebra of symmetric functions over $\mathds{Q}[t]$ which we denote as $\Lambda [t]$.

For $i$ a positive integer, we define the following power symmetric function:
$$ p_i=\sum_{j=1}^{2n} x_j^i.$$

\noindent Let $\lambda=[\lambda_1, \lambda_2, \ldots, \lambda_k]$ be a partition of $n$. The \emph{power symmetric function of $\lambda$} is defined as follows:

$$p_\lambda=p_{\lambda_1}p_{\lambda_2}\cdots p_{\lambda_k}.$$

The set $\{p_\lambda \ | \ \lambda \vdash n \}$ is a $\mathds{Q}[t]$-module basis for $\Lambda[t]$ \cite{Sagan}. Srinivasan provides an algorithm that generates another basis comprised of symmetric functions $E_\mu$ such that $E_\mu(c(\lambda))=\phi_\mu^\lambda$. Note that these formulae are not dependent on the shape $\lambda$.  See Section \ref{S:sym} for the formulae given in \cite{Srinivasan} in terms of power symmetric functions. Using the character tables for small values of $n$, we derive two other formulae for $E_\mu$ in terms of the basis of symmetric power functions. We conclude with a brief note on the standard convention when evaluating these symmetric functions over $\mathds{Q}[t]$ by content evaluation. Namely, we replace $t$ with $|\lambda|$, where $|\lambda|$ is the number of boxes in the Young diagram of $\lambda$.

Of importance to our argument in Section \ref{S:trace} is the multiplicity of the eigenvalue $\phi^{\lambda}_{\mu}$ occurring on the $\lambda$-eigenspace, which we denote as $f^{2\lambda}$. Note that, if $\tau=\phi^{\lambda}_{\mu}$ occurs on at least two Specht module, then its multiplicity is the sum of all $f^{2\lambda}$ indexed by the $\tau$-eigenspaces. The multiplicity of $\phi^{\lambda}_{\mu}$ can be obtained by computing the character degrees $\chi^{2\lambda}(1)$. It is well-known that $\chi^{2\lambda}(1)$ is the dimension of the irreducible representation of $S_{2n}$ indexed by $2\lambda$ \cite{Sagan}. Combinatorial formulae for computing the multiplicity of $\phi^{\lambda}_{\mu}$ can be obtained by applying the hook rule defined below. The \emph{hook length} of cell $(i,j)$ in  the Young tableau of partition $\lambda$, denoted $h_{2\lambda}(i,j)$, is the total number of boxes to the right and in the same row as the $(i,j)$-cell, and below and in the same column as the $(i,j)$-cell plus one. 

\begin{lemma}\label{lem:mul}
Let $2\lambda \vdash 2n$ and let $H(2\lambda)$ be the product of the $h_{2\lambda}(i,j)$ as $(i,j)$ runs over all $2n$ cells of $2\lambda$. Then 

\begin{equation} 
f^{2\lambda}=\chi^{2\lambda}(1) = \frac{(2n)!}{H(2\lambda)}. 
\end{equation}
\end{lemma}

The formulae given in the statement of Lemma \ref{lem:mul} is known as the {\it hook length formula} for $\chi^{2\lambda}(1)$.

There is also a Frobenius character formula for $\chi^{2\lambda}(1)$. Let $2\lambda = (2\lambda_1,2\lambda_2,\dots,2\lambda_k)$ and $\ell_i = 2\lambda_i + k - i$ for $i=1,\dots,k$.  Note that $\lambda_1+k-1=\ell_1 > \ell_2 > \cdots > \ell_k = 2\lambda_k$ because these are the hook lengths corresponding to the first cell of each row in the Young tableau of shape $2\lambda$.  Consequently, the Frobenius' character formula is  

\begin{equation} 
\chi^{2\lambda}(1) = \frac{(n!)}{\ell_1! \ell_2! \cdots \ell_k !} \prod_{i<j} (\ell_i - \ell_j). 
\end{equation}

In  Section \ref{S:trace}, our proof of Conjecture \ref{conj:main} for certain matrices relies on the fact that the multiplicity of most eigenvalues of these matrices is at least $\binom{2n}{3} - \binom{2n}{2}$. This lower bound follows from the fact that the dimensions of the Specht modules of $S_{2n}$ indexed by partitions of $2n$ grows quickly as we move away from extremal partitions in the dominance ordering. Therefore, few of the  irreducible characters have small degree. In fact, Meagher et.~al \cite{setpart} have characterized all irreducible representations of $S_{2n}$ with dimension less than $\binom{2n}{3} - \binom{2n}{2}$ in Lemma \ref{lem:tenspecial} below. 

\begin{lemma}\label{lem:tenspecial} \cite{setpart}
  For $n \geq 7$, if $S^{\lambda'}$ is an irreducible representation of $S_{2n}$ with
dimension less than $\binom{2n}{3} - \binom{2n}{2}$, then $\lambda'$ is one of the following partitions of $2n$:
\begin{align*}
 & [2n],  \quad [1^{2n}],  \quad [2n-1,1], \quad [2,1^{n-2}],  \quad   [2n-2,2],  \quad  [2,2,1^{2n-4}], & \\
 & [2n-2,1,1],    \quad  [3,1^{2n-3}], \quad [2n-3,3],  \quad  [2,2,2,1^{2n-6}]. &
\end{align*}
\end{lemma}

In the context of the perfect matching association scheme, since eigenspaces corresponds to the even partitions, Lemma \ref{lem:tenspecial} implies the following corollary.

\begin{coro} \cite{setpart}
\label{cor:dim}
Let $n \geqslant 7$. The two eigenspaces of the perfect matching association scheme with dimension less than $\binom{2n}{3} - \binom{2n}{2}$ are indexed by partitions $\lambda \in \{ [n], [n-1,1]\}$.  
\end{coro} 

\subsection{Quotient graphs}

We conclude this section by discussing a particular method for obtaining eigenvalues of a graph $X$. In this approach, we will construct a  directed multi-graph of order much smaller than that of $X$ whose eigenvalues are also eigenvalues of $X$. In order to do so, we first construct an equitable partition of $X$.  An \emph{equitable partition} of $X$, denoted $\pi=[C_1, C_2, \ldots, C_t]$, is a partition of $V(X)$ into $t$ sets: $C_1, C_2, \ldots, C_t$. For each pair $C_i$ and $C_j$, where $i$ and $j$ need not be distinct, each vertex of $C_i$ is adjacent to exactly $a_{i,j}$ vertices of $C_j$. Using an equitable partition, we can construct a directed multigraph known as the \emph{quotient graph} of $X$ relative to $\pi$, denoted $X / \pi$, as follows. The vertices of our directed graph are the parts of the equitable partition $\pi$. There are then precisely $a_{i,j}$ directed edges from $C_i$ to $C_j$. Note that this multigraph may contain loops and may not be symmetric. 

As per Lemma \ref{lem:quotient}, stated below, the eigenvalue of $X / \pi$ are eigenvalues of $X$. Generally, the order of $X / \pi$ is small and thus, it is computationally feasible to find its eigenvalues. Lemma \ref{lem:quotient} is a well-known result in algebraic graph theory and is instrumental in Section \ref{S:gap} in which we compute the spectral gaps of some of the graphs in our scheme. See Section 2.2 of \cite{GodMeaBook} for a proof of Lemma \ref{lem:quotient}. 

\begin{lemma} \label{lem:quotient}
If $\pi$ is an equitable partition of $X$, then eigenvalues of $X / \pi$ are eigenvalues of $X$. 
\end{lemma} 

\section{Proof of Conjecture \ref{conj:main} for \texorpdfstring{$n$}{n} sufficiently large}\label{S:trace}

In this section, we confirm Conjecture \ref{conj:main} when $n$ is sufficiently large relative to a particular parameter defined in the hypothesis of Theorem \ref{thm:maintrace} stated below. In fact, we show that the second largest eigenvalue in absolute value occurs on the $[n-1,1]$-eigenspace for several graphs in the perfect matching association scheme. As per Lemma \ref{lem:foreig}, the eigenvalue of $X_\mu$ of the $[n-1,1]$-eigenspace is always positive when $\mu$ contains at least one part of size 1. Therefore, if the second largest eigenvalue in absolute value occurs on the $[n-1,1]$-eigenspace, then this eigenvalue is also the second largest for the applicable partitions $\mu$.   

 We prove  this section's main result, Theorem \ref{thm:maintrace}, by assuming that there exists a $\lambda \vdash n$ such that $|\phi^{\lambda}_{\mu}| >| \phi^{[n-1,1]}_{\mu}|$. Then, we show that the dimension of the corresponding eigenspace must be strictly less than $\binom{2n}{3} - \binom{2n}{2}$ when $n$ is sufficiently large, thereby contradicting Corollary \ref{cor:dim}. 

\begin{theorem} \label{thm:maintrace}
Let $\mu = [n-k, \mu']$ with $\mu' \vdash k$. If $n$ is sufficiently large relative to $k$, then $\phi^{[n-1,1]}_\mu$ is the second largest eigenvalue of $X_{\mu}$ in absolute value.
\end{theorem}

Our proof of Theorem \ref{thm:maintrace} will use the trace of $A^2_{\mu}$, denoted $tr(A^2_{\mu})$. Our derivations hinge on two crucial facts. First,  since $A_{\mu}$ is the adjacency matrix of a $v_{\mu}$-regular graph on $(2n-1)!!$ vertices, it follows that $tr(A^2_{\mu})=v_{\mu}((2n-1)!!)$. Secondly, it is well-known that, for each $\lambda \vdash n$, $(\phi^{\lambda}_{\mu})^2$ is an eigenvalue of $A^2_{\mu}$ occurring with multiplicity $f^{2\lambda}$. Since the trace of a matrix is the sum of its eigenvalues, it follows that $tr(A^2_{\mu})=\sum_{\lambda} (\phi^{\lambda}_{\mu})^2 f^{2\lambda}$. These observations jointly give rise to the following equation:

\begin{equation}\label{tracequality}
v_{\mu} (2n-1)!!= \sum_{\lambda \vdash n} (\phi^{\lambda}_{\mu})^2 f^{2\lambda}.
\end{equation}

Now, if we let $n \ge 8$, $\xi$ be the largest eigenvalue in absolute value in $\{\phi_\mu^\lambda \ | \ \lambda \neq [n], [n-1,1]\}$, and $f^{2\lambda}$ be the dimension of the corresponding  $\xi$-eigenspace, then 

$$\begin{array}{rcl} 
v_{\mu} (2n-1)!! &=& \sum_{\lambda} (\phi^{\lambda}_{\mu})^2 f^{2\lambda} \\
                         &=& v_{\mu}^2 + (\phi^{[n-1,1]}_{\mu})^2 f^{[2n-2,2]} + \displaystyle{\sum_{\lambda \ne [n],[n-1,1]}} (\phi^{\lambda}_{\mu})^2 f^{2\lambda} \\
&\ge &  v_{\mu}^2 +  (\phi^{[n-1,1]}_{\mu})^2 f^{[2n-2,2]} + \xi^2 f^{2\lambda}.
\end{array}$$ 

\noindent Consequently, we have the following inequality:

\begin{equation}\label{eq:deg2}
 v_{\mu} (2n-1)!! -  v_{\mu}^2 - (\phi^{[n-1,1]}_{\mu})^2 f^{[2n-2,2]} \ge \xi^2 f^{2\lambda}. 
\end{equation}

If there exists a $\lambda \vdash n$ such that  $\lambda \not \in \{[n], [n-1, 1]\}$ and  $|\xi|=|\phi^{\lambda}_{\mu}| >| \phi^{[n-1,1]}_{\mu}|$, then Equation \ref{eq:deg2} can be used to derive an upper-bound on $f^{2\lambda}$. Namely, Equation \ref{eq:deg2} implies that, if an eigenvalue of $X_\mu$ is large enough, then its multiplicity must be small. However, Corollary \ref{cor:dim} gives a lower bound on the dimension of eigenspaces not indexed by $[n]$ and $[n-1,1]$. Therefore, we will obtain a contradiction by using Equation \ref{eq:deg2}  to derive an upper-bound on $f^{2\lambda}$ that is strictly less than the lower-bound of Corollary \ref{cor:dim} under certain conditions.  We derive this upper-bound in Lemma \ref{lem:degbou}. To prove Lemma \ref{lem:degbou}, we will require Lemma \ref{lem:sqrt} stated below. 

\begin{lemma}\label{lem:sqrt}
If $n \geq 2$, then $\frac{ (2n-1)!!}{(2n)!!} < \frac{1}{\sqrt{n+1}}$.
\end{lemma}
\begin{proof}
We prove the statement by induction. Our base case is $n=2$ which is straightforward to verify. We then assume that $\frac{ (2n-1)!!}{(2n)!!} < \frac{1}{\sqrt{n+1}}$.

It is not too difficult to see that
$$
\frac{ (2(n+1)-1)!!}{(2(n+1))!!} = \left(\frac{2n+1}{2n+2}\right) \frac{ (2n-1)!!}{(2n)!!} < \left(\frac{2n+1}{2n+2}\right)\frac{1}{\sqrt{n+1}}.
$$
Therefore, it suffices to show that
$$
\left(\frac{2n+1}{2n+2}\right) \frac{1}{\sqrt{n+1}} <  \frac{1}{\sqrt{n+2}},
$$
or equivalently, that
\begin{equation}\label{eq:bou3}
(2n+1)^2 (n+2) < (n+1)(2n+2)^2.
\end{equation}

\noindent If we expand both side of the Inequality \ref{eq:bou3}, we obtain
$$
4n^3 + 12n^2 + 9n+2  < 4n^3 + 12n^2+12n +4.
$$
This last inequality is true for all $n\geq 1$, and the statement follows. \end{proof}

\begin{lemma} \label{lem:degbou}
Let $\mu \vdash n$ and $\lambda \vdash n$ such that $\lambda \not \in \{[n], [n-1,1]\}$. If $ | \phi_\mu^\lambda | \geq | \phi_\mu^{[n-1,1]} |$, then 
$$ f^{2\lambda} \leqslant 4 n^{3/2}    \prod_i m_i! (2 \mu_i)^{m_i}.$$
\end{lemma}

\begin{proof}
By removing the negative terms on the left-hand side of Inequality \ref{eq:deg2} and letting $\phi_\mu^\lambda = \phi_\mu^{[n-1,1]}$, we obtain the following inequality
$$
f^{2\lambda} \leq \frac{ v_\mu (2n-1)!! }{\xi^2}\leq \frac{ v_\mu (2n-1)!! }{ ( \phi_\mu^{[n-1,1]})^2}. 
$$

\noindent Using Lemma \ref{lem:degree} and Lemma \ref{lem:foreig}, which give explicit formulas for $v_\mu$ and for $\phi_\mu^{[n-1,1]}$, respectively, we obtain
 
\begin{align*}
\frac{ v_\mu (2n-1)!! }{ ( \phi_\mu^{[n-1,1]} )^2 }
= \frac{   (2n-1)!!  \ 4n^2(n-1)^2  }{ v_\mu \left( (2n-1) r_1 - n \right)^2 } 
= \frac{   (2n-1)!!  \ 4n^2(n-1)^2  \ \prod_i m_i! (2 \mu_i)^{m_i} }{ (2n)!! \left( (2n-1) r_1 - n \right)^2 }.
\end{align*}
Let $r_1$ denote the number of parts of size 1 in $\mu$. Lemma \ref{lem:sqrt}, in conjunction with the fact that $\left( (2n-1) r_1 - n \right)^2 \geq (n-1)^2$ for any value of $r_1$, implies that

\begin{align*}
\frac{ (2n-1)!! \ 4n^2(n-1)^2  \ \prod_i m_i! (2 \mu_i)^{m_i} }{ (2n)!! \left( (2n-1) r_1 - n \right)^2 }
\leq 
\frac{  4 n^2  \ \prod_i m_i! (2 \mu_i)^{m_i} }{ \sqrt{n}}.
\end{align*}

\noindent Thus, if there is a $\lambda$ with $ | \phi_\mu^\lambda | \geq | \phi_\mu^{[n-1,1]} |$, then 
$f^{2\lambda} \leqslant 4 n^{3/2}  \prod_i m_i! (2 \mu_i)^{m_i}$.
\end{proof}

The crux is to now show that, under certain circumstances, the upper-bound on $f^{2\lambda}$ is in fact smaller than the lower bound given in Corollary \ref{cor:dim}. Before we do so, we will require the following proposition which bounds the degree of all graphs in our scheme. 

\begin{prop}\label{propBigDegree}
Let $n\geq 2$. Then the maximum value of $v_\mu $ is $(2n-2)!!$ and occurs for $\mu = [n]$. Further, the minimum value of $v_{\mu}$ is $1$ and occurs for $\mu = [1^{n}]$.
\end{prop}

\begin{proof}
Let $\mu= [\mu_1, \mu_2,\ldots, \mu_t] \vdash n$. Since $2\mu_1+2\mu_2+\cdots + 2\mu_t =2n$, $\prod_{i=1}^t 2\mu_i \geqslant 2n$. Therefore, we have that $\prod_i m_i! (2 \mu_i)^{m_i} \geqslant 2n$ and it follows from Lemma \ref{lem:degree} that
$$
v_\mu = \frac{2^n n!}{\prod_i m_i! (2 \mu_i)^{m_i}} \leq \frac{2^n n!}{2n} = (2n-2)!!.
$$
Moreover, by Lemma \ref{lem:degree} one can see that $v_\mu=(2n-2)!!$ if and only if $\mu=[n]$. The statement about the minimum value of $v_{\mu}$ being $1$ is clear and $A_{[1^{n}]}$ is the identity matrix which attains it.
\end{proof}

\begin{proof}[Proof of Theorem \ref{thm:maintrace}]
Let $\mu = [n-k,\mu']$ with $\mu'\vdash k$. Suppose that there exists some $\lambda \vdash n$ such that a $\lambda \notin \{[n], [n-1,1]\}$ and $|\phi_{\mu}^{\lambda}|\geq |\phi_{\mu}^{[n-1,1]}|$. Thus, by Lemma \ref{lem:degbou}, we have 
$$
f^{2\lambda}\leq 4 n^{3/2}  \prod_i m_i! (2 \mu_i)^{m_i}.
$$

\noindent Proposition~\ref{propBigDegree} implies that for any $n$ and every partition $\mu \vdash n$, we have
$2n \leq \prod_{i}m_i!(2\mu_i)^{m_i}\leq (2n)!!$. If $n\gg k$, then in particular $n>2k$, and consequently, multiplicity $m_1$ of $n-k$ in $\mu$ has to be $1$. Combining all of these we have
$$
\prod_i m_i! (2 \mu_i)^{m_i}  = 1 (2n-2k)  \prod_{i>1}  m_i! (2 \mu_i)^{m_i} \leq 2(n-k) (2k!!),
$$
and thus, $f^{2\lambda}\leqslant 8 n^{3/2} (n-k) (2k)!! $. However, since $\lambda \not\in \{ [n], [n-1,1]\}$, Corollary \ref{cor:dim} implies that  $f^{2\lambda} \geqslant \binom{2n}{3}  - \binom{2n}{2}$. On the other hand, if $n\gg k$, then
\begin{align} \label{eq:largenk}
\binom{2n}{3}  - \binom{2n}{2} = \frac{n(2n-1)(2n-5)}{3}> 8 n^{3/2} (n-k) (2k)!!,
\end{align}
which is a contradiction.
\end{proof}

When $k=1$, Inequality \ref{eq:largenk} holds for all $n \geqslant 6$. This gives rise to the following corollary by Lemma \ref{lem:degree} and Lemma \ref{lem:foreig}. 

\begin{coro}\label{cor:2n-2}
For all $n \ge 6$, the second largest eigenvalue of $A_{[n-1,1]}$ is 
$$\phi^{[n-1,1]}_{[n-1,1]} = 2^{n-3} (n-2)!.$$ 
\noindent The spectral gap of $A_{[n-1,1]}$ is $2^{n-3}(2n-1)((n-2)!).$
\end{coro}

In summary, we have confirmed Conjecture \ref{conj:main} for partitions $\mu=[\mu_1, \ldots, \mu_t]$ when $\mu_1$ is large enough. In fact, Theorem \ref{thm:maintrace} also implies that the second largest eigenvalue in absolute value of many graphs in the perfect matching association scheme occurs on the $[n-1,1]$-eigenspace. However, it is not always the case that this eigenvalue occurs on the $[n-1, 1]$-eigenspace. For example, there are several graphs in for which this eigenvalue occurs on the $[n-2,2]$-eigenspace (see for example, Table~\ref{tab_for_n=6} and Table~\ref{tab_for_n=7}). We leave it as an open question to determine on which eigenspace of the perfect matching association scheme does the second largest eigenvalue in absolute value occurs. In the next section, we confirm Conjecture \ref{conj:main} for certain graphs indexed by partitions $\mu=[n-k, \mu']$ where $n-k$ is very small.

\section{Proof of Conjecture \ref{conj:main} for certain matrices using symmetric functions }\label{S:sym}

In this section, we confirm Conjecture~\ref{conj:main} for 
$$
\mu \in \bigl\{[2,1^{n-2}],\ [3,1^{n-3}],\ [2,2,1^{n-4}],\ [4,1^{n-4}],\ [3,2,1^{n-5}],\ [5,1^{n-5}]\bigr\}.
$$
To achieve this, we adopt an approach different from that of Section~\ref{S:trace}. 
Specifically, we make use of expressions for $\phi^{\lambda}_{\mu}$ in terms of power-sum symmetric functions obtained in~\cite{Srinivasan}, and derive two additional expressions using the results from~\cite{Srinivasan} together with known character tables. These formulas, combined with an inductive argument, will be used to prove Conjecture~\ref{conj:main}.

As stated in Section~\ref{S:term}, Srinivasan provides a formula for $E_{\mu}$, the character values of $A_{\mu}$ in terms of power-sum symmetric functions. When evaluated on the content of the Young tableau of shape $2\lambda$, this yields $ E_{\mu}\bigl(c(2\lambda)\bigr) = \phi^\lambda_\mu$. For the sake of completeness, we include here some details of how $E_{\mu}$ can be expressed in terms of power-sum symmetric functions following~\cite{Srinivasan}.

Let $\ell(\mu)$ denote the number of parts of a partition $\mu$. Let us define a partial order on the set of all partitions as follows: $\mu \leq \lambda$ if either $|\mu| < |\lambda|$ or $|\mu| = |\lambda|$ and $\mu$ can be obtained from $\lambda$ by partitioning the parts of $\lambda$ into disjoint blocks and summing the parts in each block. Let $\mu$ be a partition with all parts $\geq 2$, and let $\overline{\mu}$ denote the partition obtained from $\mu$ by subtracting $1$ from each part.  It follows from \cite[Theorem~4.3]{Srinivasan} that
$$
E_{\mu} \;=\; \sum_{\lambda \leq \overline{\mu}} b_{\mu}^{\lambda}(t)\, p_{\lambda},
$$
where $b_{\mu}^{\lambda}(t) \in \mathbb{Q}[t]$ and
$\deg (b_{\mu}^{\lambda}(t)) \;\leq\; |\overline{\mu}| - |\lambda| + \ell(\overline{\mu}) - \ell(\lambda)$. By using the above result Srinivasan provided the formulae for the first four symmetric functions in the Table \ref{Tab:Formulae} (see \cite[Example 4.6]{Srinivasan}.) Adding to the work of Srinivasan, we derive formulae for $\mu \in \{[3,2, 1^{n-5}], [5,1^{n-5}]\}$. 

\begin{table} [htpb]
\begin{center}
\begin{tabular}{|l |l|}
\hline
 $A_{\mu}$ & $E_\mu$ \\[2pt] \hline 
 $A_{[2,1^{n-2}]}$ & $ \frac{p_1}{2} - \frac{t}{4}$\\[2pt] \hline
$A_{[3,1^{n-3}]}$, & $\frac{p_2}{2} - p_1 + \frac{3t-t^2}{4}$\\[2pt] \hline
$A_{[2,2,1^{n-4}]}$& $\frac{p_1^2}{8}-\frac{3p_2}{4} + \frac{(10-t)p_1}{8} + \frac{9t^2-24t}{32}$\\[2pt]  \hline
$A_{[4,1^{n-4}]} $ &  $\frac{p_3}{2} - \frac{9p_2}{4} + \frac{(11-2t)p_1}{2} + \frac{8t^2-23t}{8}$\\[2pt]  \hline
$A_{[3,2, 1^{n-5}]}$ &  $-2p_3+\frac{1}{4}p_1p_2+(\frac{60-t}{8})p_2-\frac{1}{2}p_1^2+\frac{29t-120-t^2}{8}p_1+\frac{116t-47t^2+t^3}{16}$\\[2pt]  \hline
$A_{[5,1^{n-5}]}$ &  $\frac{p_4}{2} -4p_3+\frac{40-3t}{2}p_2-p_1^2+(7t-34)p_1+\frac{217t-96t^2+5t^3}{12}$\\[2pt] \hline
\end{tabular}
\vspace{1ex}
\caption{Formulae for the symmetric functions to compute eigenvalues of certain matrices in the perfect matching association scheme}
\label{Tab:Formulae}
\end{center}
\end{table}

We illustrate the process by determining $E_{[3,2,1^{n-5}]}$, while for $E_{[5,1^{n-5}]}$, we leave the details to the reader.\footnote{All computations for both of the symmetric functions can also be found in the SageMath file available on: \url{https://github.com/Himanshugupta23/Perfect-Matching-Association-Scheme}.} As described above, by \cite[Theorem~4.3]{Srinivasan}, the symmetric function  $E_{[3,2,1^{n-5}]}$ lies in the $\mathbb{Q}[t]$-span of the power-sum symmetric functions $p_{3}$, $p_{1}p_{2}$, $p_{2}$, $p_{1}^{2}$, $p_{1}$, and $p_{0}$. To determine the coefficients, we use the columns of the tables of eigenvalues corresponding to the partition $[3,2,1^{n-5}]$ for $5 \leq n \leq 8$, and solve systems of linear equations whose coefficient matrices are formed by the corresponding columns of $p_{3}$, $p_{1}p_{2}$, $p_{2}$, $p_{1}^{2}$, $p_{1}$, and $p_{0}$ for $5 \leq n \leq 8$. Since the degrees of the polynomial coefficients are bounded, we have enough points to apply the Lagrange interpolation to recover these coefficients exactly. 

The details are as follows. For $n=5$, the column corresponding to $[3,2]$ is 
$[160,-20,$ $20,-4,-10,10,-20]^{\top}$ (see Table \ref{tab_for_n=5}). Solving the linear system with respect to power-sum symmetric functions for $2n=10$, we get 
$$E_{[3,2]} = -2p_3 + \frac14 p_1p_2 + \frac{25}{4} p_2 - \frac12 p_1^2 + \frac{35}{4} p_1 - \frac{635}{4} p_0. $$ 
For $n=6$, the column corresponding to $[3,2,1]$ is 
$[960,80,24,-60,120,0,12,-60,0,20,\linebreak -120]^{\top}$ (see Table \ref{tab_for_n=6}),
which can be expressed as 
$$E_{[3,2,1]} = -2p_3 + \frac14 p_1p_2 + 6 p_2 - \frac12 p_1^2 + \frac{21}{2} p_1 - 228 p_0. $$
For $n=7$, the column corresponding to $[3,2,1^{2}]$ is 
$[3360,760,148,-56,228,-42,-78,\linebreak 84,-60,12,52,-60,28,-20,-420]^{\top}$ (see Table \ref{tab_for_n=7}), 
which corresponds to 
$$E_{[3,2,1^{2}]} = -2p_3 + \frac14 p_1p_2 + \frac{23}{4} p_2 - \frac12 p_1^2 + \frac{45}{4} p_1 - \frac{1211}{4} p_0. $$ 
For $n=8$, the column corresponding to $A_{[3,2,1^{3}]}$ is 
$[8960,2960,848,440,512,-28,\linebreak -184,608,132,-132,-72,-52,0,96,-60,68,80,-160,-28,32,-220,-1120]^{\top}$, which reduces to the linear combination 
$$ E_{[3,2,1^{3}]} =  -2p_3 + \frac14 p_1p_2 + \frac{11}{2} p_2 - \frac12 p_1^2 + 11 p_1 - 380 p_0. $$ 
We can now interpolate the coefficients that depend on $t$ from their four values at $t=10,12,14,16$.  The coefficient of $p_2$ in $E_{[3,2,1^{n-5}]}$ will be $\frac{1}{8}(60-t)$, the coefficient of $p_1$ will be $\frac{1}{8}(-120+29t-t^2)$, and the coefficient of $p_0$ will be $\frac{1}{16}(116t-47t^2+t^3)$. Since the $E_{[3,2,1^{n-5}]}$ have unique expressions as $\mathbb{Q}[t]$-linear combinations of the $p_{\lambda}$'s, these expressions must hold for every choice of $2n \ge 10$.
 
In Proposition \ref{prop:E_2} below, we use the expression for $E_{\mu}$ to prove Conjecture \ref{conj:main} for $\mu=[2, 1^{n-2}]$. Because the expression for $E_{\mu}$ is much simpler for this particular relation, the proof of Proposition \ref{prop:E_2} is much less involved than the proofs given for other partitions $\mu$. As already noted in Section \ref{S:term}, Diaconis and Holmes \cite{DiaHol} studied the spectrum of $A_{[2,1^{n-2}]}$ in the context of random walks on the corresponding graph.

\begin{prop} \label{prop:E_2}
For all $n \ge 3$, the second largest eigenvalue of $A_{[2,1^{n-2}]}$ is 
\begin{align*}\phi^{[n-1,1]}_{[2,1^{n-2}]}= n^2 -3n + 1.
\end{align*}
The spectral gap of the graph corresponding to $A_{[2,1^{n-2}]}$ is $2n-1$.
\end{prop}

\begin{proof} 
From Table \ref{Tab:Formulae}, we see that 
\begin{align*}\phi^{\lambda}_{[2,1^{n-2}]} = E_{[2,1^{n-2}]}(c(\lambda))= \frac{1}{2}p_1(c(\lambda)) - \frac{1}{4}(2n).
\end{align*}

Recall that $p_1$ is the sum of the content of $2\lambda$. Observe that $p_1(\lambda)$ is an increasing function in the dominance ordering of partitions of $n$. 

As a consequence of the above observation, the value of $E_{[2, 1^{n-2}]}(c(\lambda))$ increases with the dominance ordering of the partitions of $n$.  Since $[n-1,1]$ is the second largest partition in the dominance ordering of partitions, the second largest eigenvalue will occur on the $[n-1,1]$-eigenspace, as claimed. 

As for the spectral gap, the valency is given by
\begin{align*}
v_{[2, 1^{n-2}]} = E_{[2, 1^{n-2}]}(c([n])) = \frac{1}{2} \left(\sum_{i=1}^{2n-1} i \right) - \frac{2n}{4} = \frac{(2n-1)(2n)}{4} - \frac{2n}{4} = n^2 - n; 
\end{align*}
and that 
\begin{align*}
\phi_{[2,1^{n-2}]}^{[n-1,1]} &= E_{[2, 1^{n-2}]}(c([n-1,1]))= \frac{1}{2} \left(\left(\sum_{i=1}^{2n-3} i\right) + (-1+0)\right) - \frac{2n}{4} \\
&= \frac{(2n-3)(2n-2)}{4} - \frac{1}{2} - \frac{n}{2} = n^2 -3n + 1.
\end{align*} 

Therefore, the spectral gap is
$v_{[2,1^{n-2}]}-\phi_{[2,1^{n-2}]}^{[n-1,1]}=2n-1.$ \end{proof}

Next, we describe how we intend to use the formulae of Table \ref{Tab:Formulae} to prove Conjecture \ref{conj:main} for the remaining partitions listed in the beginning of this section. Since we will prove Conjecture \ref{conj:main} by inducting on $n$, we first describe three ways to construct a partition of $n+1$ from a partition of $n$. Let $\lambda=(\lambda_1, \lambda_2, \ldots, \lambda_k)$ be a partition of $n$. Define $\lambda^+=(\lambda_1^+, \lambda_2^+, \ldots, \lambda_k^+)$ to be a partition of $n+1$ obtained from $\lambda$, in one of the following three ways:
\begin{enumerate} 
\item by adding a box to the first row of the Young tableau of $\lambda$:  
$$\lambda^+=(\lambda_1+1, \lambda_2, \ldots, \lambda_k);$$
\item by adding a box to the $i$-th row, with $2 \le i \le k$ and provided $\lambda_{i-1} > \lambda_i$:  
$$\lambda^+=(\lambda_1, \lambda_2, \ldots, \lambda_{i-1}, \lambda_i+1, \ldots \lambda_k^+); \text{ or}$$  
\item by adding a new row consisting of a single box: 
$$\lambda^+=(\lambda_1, \lambda_2, \ldots, \lambda_k, 1).$$ 
\end{enumerate}

Note that, although Case 1 and Case 2 appear to describe the same transformation, we will need to consider these two cases separately because our approach is different for each case. Recall that the sizes of the parts of a partition are always given in decreasing order. For that reason, $\lambda^+$ can only be obtained from $\lambda$ by  one of the three transformations given above.  See Figure \ref{fig:tablambdaplus} for an illustration of the second transformation. 

\ytableausetup{mathmode, boxsize=1.5em}
\begin{figure} [htpb]
\begin{center}
\begin{subfigure}{0.49\textwidth}
\centering
\begin{ytableau}
0 & 1&2 &3&4 &5\\
-1 & 0 & 1 &2\\
-2 & -1 
\end{ytableau}
\caption{Young tableau for partition $2\lambda=2[3,2,1]$ with its content.}
\end{subfigure}
\begin{subfigure}{0.49\textwidth}
\centering
\begin{ytableau}
0 & 1&2 &3&4 &5\\
-1 & 0 & 1 &2 &3 & 4\\
-2 & -1 
\end{ytableau}
\caption{Young tableau for partition $2\lambda^+=2[3,3,1]$ with its content.}
\end{subfigure}
\end{center}
\caption{Illustrating an admissible construction of $\lambda^+$ from $\lambda$.}
\label{fig:tablambdaplus}
\end{figure}

Now, we introduce some notation to simplify our computations. If $f \in \Lambda [t]$ is a symmetric function evaluated on $c(\lambda)$, and $\lambda^+$ is obtained from $\lambda$ as described above, we define 
$$\Delta f (\lambda):= f(c(\lambda^+))-f(c(\lambda)).$$ 
Namely, we will write $\Delta E_\mu(c(2\lambda))$ to refer to $\phi_\mu^{\lambda^+}-\phi_\mu^{\lambda}$. This parameter describes the changes in the eigenvalues when moving from $n$ to $n+1$ through one of the three possible transformations. For convenience, we write $E(\lambda)$ instead of $E(c(2\lambda))$, and $p_i(\lambda)$ (or simply $p_i$) instead of $p_i(c(2\lambda))$ when it is clear on which partition $\lambda$ is the symmetric function being evaluated on.

The purpose of Lemma \ref{induction_lemma} is to show that, in order to prove that the second largest eigenvalue is attained at the $2[n-1,1]$-eigenspace, it suffices to establish a base case and to show that the increment in the corresponding eigenvalues from $[n-1,1]$ to $[n,1]$ is always greater than any increment from $\lambda$ to $\lambda^{+}$ for $\lambda \neq [n]$. We will use Lemma \ref{induction_lemma} in the rest of the theorems of this section and the fact that 
we know Conjecture \ref{conj:main} holds for all $n \leqslant 15$. Thus, the base case of our induction will always be $n=15$. 

\begin{lemma}\label{induction_lemma}
Let $\mu=[\mu_1,\ldots,\mu_t]$ be a partition of $n_0$, and for any $n\geq n_0$ define $\mu(n) = [\mu_1,\ldots,\mu_t,1^{n-n_0}]$ a partition of $n$. Assume that the second largest eigenvalue of $A_\mu$ is $\phi_{\mu}^{[n_0-1,1]}$, and that 
\begin{align}\label{eqn_within_lemma}
\Delta E_{\mu(n)}(\lambda) \leq E_{\mu(n+1)}([n,1])-E_{\mu(n)}([n-1,1])    
\end{align}
for all $n\geq n_0$ and for every partition $\lambda$ of $n$ with $\lambda\neq [n]$. Then the second largest eigenvalue of $A_\mu(n)$ is $\phi_{\mu(n)}^{[n-1,1]}$ for all $n \geq n_0$.
\end{lemma}
\begin{proof}

We prove by induction on $n$.  The base case $n=n_0$ holds by assumption. Assume the statement holds for some $n \geq n_0$; that is, we assume $E_{\mu(n)}(\lambda) \;\leq\; E_{\mu(n)}([n-1,1])$ for all partitions $\lambda \neq [n]$. By the hypothesis on $\Delta E_{\mu(n)}$, this implies $E_{\mu(n+1)}(\lambda^{+}) \;\leq\; E_{\mu(n+1)}([n,1])$ for a partition $\lambda^+$ of $n+1$. Since every partition of $n+1$ other than $[n+1]$ arises in this way, it follows that the second largest eigenvalue of $A_\mu(n+1)$ is $\phi_{\mu(n+1)}^{[n,1]}$. Hence, the statement follows by induction on $n$. \end{proof}

In what follows, all of our proofs to establish Equation \ref{eqn_within_lemma} for different partitions $\mu$ are divided into three cases that correspond to each of the three transformations of $\lambda$ to $\lambda^+$ described above. Below, we describe the effect of these three transformation when applied to the Young tableau of $2\lambda$:
\begin{enumerate} 
 \item if $\lambda_1^+=\lambda_1+1$, then two boxes are added to the first row of the Young tableau of $2\lambda$; 
\item if $\lambda_i^+=\lambda_i+1$ for $2 \le i \le k$ when $\lambda_{i-1} > \lambda_i$, then two boxes are added to an existing row of the Young tableau of $2\lambda$ that is not the first row; and
\item if $\lambda^+_{k+1} = 1$, then a new row with two boxes is added to the Young tableau of $2\lambda$.  
\end{enumerate}

Let $\lambda_i^+=\lambda_i+1$ for $1 \le i \le k+1$ where $\lambda_{i-1} > \lambda_i$ and $\lambda_{k+1}=0$. Notice that the first box of row $i$ in $2\lambda$ is filled with $-(i-1)$ and thus, the content of the last box of $2\lambda$ in row $i$ is filled with $-(i-1)+(2\lambda_i-1)$.  Consequently, the two new boxes of $\lambda^+$ are filled with $-(i-1)+2\lambda_i$ and $-(i-1)+2\lambda_i + 1$. This observation allows use to compute $\Delta f(\lambda)$ for several power symmetric functions (see Table \ref{Tab:FormulaeDelta}) as a function of $\lambda_i$ and the row $i$ that is modified while moving from $\lambda$ to $\lambda^+$. We will frequently need these formulae in our computations.

\begin{table} [htpb]
\begin{center}
\begin{tabular}{| l | l |}
\hline
$\Delta p_1$ & $-2(i-1)+4\lambda_i+1$ \\ \hline
$\Delta p_2$ & $2i^2-6i+5-8i\lambda_i+12\lambda_i+8\lambda_i^2$ \\ \hline
$\Delta p^2_1$ & $(8\lambda_i-4i+6)p_1+16\lambda_i^2-16i\lambda_i+24\lambda_i+4i^2-12i+9$ \\ \hline
$\Delta p_3$ &  $-2i^3+12i^2\lambda_i -24i \lambda_i^2 + 16 \lambda_i^3 +9i^2 -36i\lambda_i +36\lambda_i^2 - 15i + 30\lambda_i+9$ \\ \hline
\end{tabular}
\vspace{1ex}
\caption{Formulae for increments of certain power symmetric functions}.
\label{Tab:FormulaeDelta}
\end{center}
\end{table}

We are now prepared to prove Conjecture \ref{conj:main} for several matrices, starting with $A_{[3,1^{n-3}]}$.

\begin{theorem} 
\label{thm:E_3}
For all $n \ge 5$, the second largest eigenvalue of $A_{[3,1^{n-3}]}$ is 
\begin{align*}
\phi^{[n-1,1]}_{[3,1^{n-3}]} = \frac{4}{3} n^3 - 8 n^2 + \frac{38}{3} n -4.
\end{align*}
The spectral gap of the graph corresponding to $A_{[3,1^{n-3}]}$ is  $4n^2 - 10 n + 4.$
\end{theorem} 

\begin{proof} As already noted, the statement holds for all $n \leq 15$. Therefore, by Lemma \ref{induction_lemma}, to prove the statement it suffices to show that 
$$
\Delta E_{[3,1^{\,n-3}]}(\lambda) 
\;\leq\; 
E_{[3,1^{\,n-2}]}([n,1]) - E_{[3,1^{\,n-3}]}([n-1,1])
$$
for all $n\geq 15$ and for every partition $\lambda$ of $n$ with $\lambda \neq [n]$.

Let $\lambda = [\lambda_1, \dots, \lambda_k]$ be a partition of $n$ and suppose that $\lambda^+$ is constructed from $\lambda$ in one of the three admissible ways described above this theorem. We derive an upper-bound on $\Delta E_{[3,1^{n-3}]}(\lambda)$ for each case. Central to our argument is the character formula $E_{[3,1^{n-3}]} = \frac{p_2}{2} - p_1 + \frac{3t-t^2}{4}$. Refer to Table \ref{Tab:FormulaeDelta} for the formulas for $\Delta p_1$ and  $\Delta p_2$. Additionally,  we have $\Delta \frac{3t-t^2}{4}= \frac{3(2n+2)-(2n+2)^2}{4}- \frac{3(2n)-(2n)^2}{4}=\frac{1-4n}{2}.$ Therefore, we obtain
\begin{align}\label{eq:E3}
    \Delta E_{[3,1^{n-3}]}(\lambda) &=  \frac{1}{2}\Delta p_2 - \Delta p_1 + \Delta \frac{3t-t^2}{4} \nonumber\\
    &=4\lambda_i^2 + 2\lambda_i + i^2-4i\lambda_i-i-2n,
\end{align}
where $i=1,\ldots,k+1$ and $\lambda_{k+1}=0$.

\noindent Case 1: Let $\lambda_1^+=\lambda_1+1$, that is, $i=1$. We see that

\begin{equation} \label{eq:E3c1}
\Delta E_{[3,1^{n-3}]}(\lambda)= 4 \lambda_1^2 - 2 \lambda_1 - 2n.
\end{equation}

Letting $\lambda_1=n-1$ in Equation \ref{eq:E3c1}, we obtain
  
\begin{equation} \label{eq:E3c1n}
E_{[3,1^{n-2}]}([n,1])-E_{[3,1^{n-3}]}([n-1,1])=4(n-1)^2 - 2(n-1) - 2n = 4n^2 -12n +6.
\end{equation}
 
Next, we derive an upper-bound on $\Delta E_{[3,1^{n-3}]}(\lambda)$ when $\lambda \not \in \{[n], [n-1,1]\}$. If $\lambda \not \in \{[n], [n-1,1]\}$, then $\lambda_1 \leqslant  n-2$. Hence, substituting  $\lambda_1 \leqslant  n-2$ in Equation \ref{eq:E3c1}, we obtain that  $\Delta E_{[3,1^{n-3}]}(\lambda) \leqslant 4 n^2 - 20 n + 20$.

Therefore, when $n \geqslant 15$, we have $\Delta E_{[3,1^{n-3}]}(\lambda)< E_{[3,1^{n-2}]}([n,1])-E_{[3,1^{n-3}]}([n-1])$. 

\noindent  Case 2: Let $\lambda_i^+=\lambda_i+1$ for $i\geqslant 2$. Observe that $i \le k \le n$, where $k$ is the number of parts of $\lambda$, and $\lambda_i \le \frac{n}{2}$. Therefore, by removing  the nonpositive terms of Equation \ref{eq:E3}, and replacing $i$ and $\lambda_i$ with $n$ and $\frac{n}{2}$, respectively, we obtain that 

\begin{align*}
    \Delta E_{[3, 1^{n-3}]}(\lambda) \leqslant 4\left(\frac{n}{2}\right)^2 + 2\left(\frac{n}{2}\right)+n^2 = 2n^2+n.
\end{align*}

However, in Equation \ref{eq:E3c1n}, we computed $E_{[3,1^{n-2}]}([n,1])-E_{[3,1^{n-3}]}([n-1])$, and thus, we see that  $ \Delta E_{[3, 1^{n-3}]}(\lambda) < E_{[3,1^{n-2}]}([n,1])-E_{[3,1^{n-3}]}([n-1])$ when $\lambda \not\in \{[n], [n-1,1]\}$ and $n \geqslant 15$.  \\

\noindent Case 3: Let $\lambda^+_{k+1}=1$, that is, $i=k+1$ and $\lambda_k=0$. By using Equation \ref{eq:E3}, this means that
\begin{equation} \label{eq:ckE3}
\Delta E_{[3, 1^{n-3}]}= k^2 +k - 2n.
\end{equation}

\noindent Note that $k \leqslant n$.  Substituting the maximum value for $k$ in Equation \ref{eq:ckE3} and removing the negative term, we see that $ \Delta E_{[3, 1^{n-3}]}(\lambda) \leqslant n^2$. By Equation \ref{eq:E3c1n}, we have  $ \Delta E_{[3, 1^{n-3}]}(\lambda)\leqslant  E_{[3,1^{n-2}]}([n,1])-E_{[3,1^{n-3}]}([n-1])$ when $\lambda \not\in \{[n], [n-1,1]\}$ and $n \geqslant 15$.  

In all three cases, we have shown that  $ \Delta E_{[3, 1^{n-3}]}(\lambda) \leqslant E_{[3,1^{n-2}]}([n,1])-E_{[3,1^{n-3}]}([n-1])$ for all $\lambda \vdash n$ such that $\lambda \not\in \{[n-1,1], [n]\}$ and $n\geq 15$. In conclusion, we see that  $\phi^{[n-1,1]}_{[3, 1^{n-3}]}$ is the second largest eigenvalue of $A_{[3, 1^{n-3}]}$. 

Next, we compute the spectral gap using Lemma \ref{lem:degree} and Lemma \ref{lem:foreig} to obtain

\begin{align*}
\phi^{[n]}_{[3, 1^{n-3}]}&=\frac{4n(n-1)(n-2)}{3} && &\phi^{[n-1,1]}_{[3, 1^{n-3}]}&=\frac{2(2n-1)(n-2)(n-3)-2n(n-2)}{3}\\
&=\frac{4}{3} n^3 - 4 n^2 + \frac{8}{3} n ; && &&=\frac{4}{3} n^3 - 8 n^2 + \frac{38}{3} n -4.
\end{align*}

\noindent Therefore, the spectral gap is $\phi^{[n]}_{[3, 1^{n-3}]}- \phi^{[n-1,1]}_{[3, 1^{n-3}]}=4n^2-10n+4$. \end{proof}

We now proceed with proof of Conjecture \ref{conj:main} for the relation $\mu=[2,2,1^{n-4}]$.  Before we do so, we will required the following lemma on the properties of $p_1(\lambda)$. This lemma will also be required for the proof of the rest of the theorems in this section. 

\begin{lemma} \label{lem:p1}
Let $\lambda \vdash n$ such that $\lambda=[\lambda_1, \lambda_2, \ldots, \lambda_k]$. 
\begin{enumerate}
\setlength\itemsep{1em}
\item If $\lambda_1>\frac{n}{2}$, then $p_1(\lambda)>\frac{n^2}{4}$
\item For all $\lambda \vdash n$, we have  $-n^2+2n\leqslant p_1(\lambda)$.  
\end{enumerate}
\end{lemma}

\begin{proof}
We note that $p_1(\lambda)$ is an increasing function in the dominance ordering of partitions of $n$. Therefore, it suffices to prove the statement for the lowest applicable partition in this ordering, namely $\lambda=[\frac{n}{2}+1, 1^{\frac{n-2}{2}}]$ when $n$ is even and $\lambda=[\frac{n+1}{2}, 1^{\frac{n-1}{2}}]$ when $n$ is odd:
\begin{align*}
p_1\left(\left[\frac{n}{2}+1, 1^{\frac{n-2}{2}}\right]\right)&=\sum_{i=1}^{n+1}i-2\left(\sum_{i=1}^{\frac{n-2}{2}}i\right)+\frac{n-2}{2}\\
&= \binom{n+2}{2}-2\binom{\frac{n}{2}}{2}+\frac{n}{2}-1\\
&=\frac{n^2}{4}+\frac{5n}{2};\\
p_1\left(\left[\frac{n+1}{2}, 1^{\frac{n-1}{2}}\right]\right)&=\sum_{i=1}^{n}i-2\left(\sum_{i=1}^{\frac{n-1}{2}}i\right)+\frac{n-1}{2}\\
&= \binom{n+1}{2}-2\binom{\frac{n+1}{1}}{2}+\frac{n-1}{2}\\
&=\frac{n^2}{4}+n-\frac{1}{4}.
\end{align*}

We prove the second equality by an argument similar to that given for the first statement, it suffices to show that $-n^2+2n\leqslant p_1([1^n])$. It is not too difficult to see that $p_1([1^n]) = -2 \binom{n}{2}+n-1+1=-n^2+2n.$\end{proof}

\begin{theorem} 
\label{thm:E_44}
For all $n \ge 6$, the second largest eigenvalue of $A_{[2,2,1^{n-4}]}$ is 
$$\phi^{[n-1,1]}_{[2,2,1^{n-4}]} = \frac12 n^4 - 5n^3 + \frac{33}{2} n^2 -20n +6.$$
The spectral gap of the graph corresponding to $A_{[2,2,1^{n-4}]}$ is 
$$ \phi^{[n]}_{[2,2,1^{n-4}]} - \phi^{[n-1,1]}_{[2,2,1^{n-4}]} = 2n^3-11n^2+17n-6.$$
\end{theorem}

\begin{proof} Since Conjecture \ref{conj:main} has already been verified for all $n \leq 15$, Lemma \ref{induction_lemma} reduces the proof of the statement to showing that
$$
\Delta E_{[2,2,1^{\,n-4}]}(\lambda) 
\;\leq\; 
E_{[2,2,1^{\,n-3}]}([n,1]) - E_{[2,2,1^{\,n-4}]}([n-1,1])
$$
for all $n \geq 15$ and for every partition $\lambda$ of $n$ with $\lambda \neq [n]$. For that, we know
$$\phi^{\lambda}_{[2,2,1^{n-4}]} = E_{[2,2,1^{n-4}]}(\lambda) = \frac{p_1^2}{8}-\frac{3p_2}{4} + \frac{(10-t)p_1}{8} + \frac{9t^2-24t}{32}.$$  
\noindent This means that
\begin{equation} \label{eq:E422c1}
\Delta E_{[2,2,1^{n-4}]}(\lambda) = \frac{1}{8} \Delta(p_1^2)-\frac{3}{4} \Delta(p_2) + \frac{1}{8} \Delta((10-t)p_1) + \frac{1}{32}\Delta(9t^2-24t).
\end{equation}
When $\lambda_i^+=\lambda_i$, for $1 \leqslant i \leqslant k$, we have 
\begin{align*}
\Delta((10-t)p_1) &= (10 - (2n+2))(p_1 + \Delta p_1) - (10-2n)p_1 \\ 
                &= -2p_1 + (8-2n)\Delta p_1 \\
                &= -2p_1 + (8-2n)(4\lambda_i -2i +3) \\
                &= -2p_1 + 32 \lambda_i -8n\lambda_i -16i + 4ni + 24 -6n, \\
\Delta(9t^2-24t) &= [9(2n+2)^2-24(2n+2)] - [9(2n)^2 - 24 (2n)] \\ 
                &= 9(8n+4) -24(2) \\
                &= 72n - 12.
\end{align*}
Refer to Table \ref{Tab:FormulaeDelta} for the formula of $\Delta p_2$ and $\Delta p_1^2$. With the assistance of a computer, we then simplify Equation \ref{eq:E422c1}:
\begin{equation} \label{equ:42}
\Delta E_{[2,2,1^{n-4}]}(\lambda) = \left(\lambda_i - \frac{i}{2}+\frac{1}{2}\right)p_1 + 4i\lambda_i+\frac{1}{2}ni +\frac{3}{2}n-i(i-1)-4\lambda_i^2 -n\lambda_i -2\lambda_i,
\end{equation}
where $i=1,\ldots,k+1$ and $\lambda_{k+1}=0$.

\noindent Case 1: Let $\lambda_1^+=\lambda_1+1$, that is, $i=1$. From Equation \ref{equ:42}, we have  
\begin{equation} \label{eq:42c1}
\Delta E_{[2,2,1^{n-4}]}(\lambda) = \lambda_1 p_1 -4 \lambda_1^2 -(n-2)\lambda_1+2n.
\end{equation}

When $\lambda=[n-1,1]$, then $\lambda_1=n-1$ and $p_1 = -1+\frac{1}{2}(2n-3)(2n-2)$. Thus, we obtain 
\begin{align}\label{eqn_this}
    E_{[2,2,1^{n-3}]}([n,1])-E_{[2,2,1^{n-4}]}([n-1,1]) = 2n^3 - 12n^2 +20n -8.
\end{align}

Next, we note that, when $\lambda \not\in \{[n], [n-1,1]\}$, we have $\lambda_1 \leqslant n-2$ and $p_1(\lambda) \leqslant 2 +\frac{1}{2} (2n-5)(2n-4) = 2n^2-9n+12$. If we remove the second and third terms of Equation \ref{eq:42c1} and substitute the upper-bound for $\lambda_1$ and $p_1$, we obtain  the following upper-bound on $\Delta E_{[2,2,1^{n-4}]}(\lambda)$ for $\lambda \not\in \{[n], [n-1,1]\}$:
\begin{align*} 
\Delta E_{[2,2,1^{n-4}]}(\lambda) \le (n-2)(2n^2-9n+12)+2n= 2n^3 - 13 n^2 + 32n -24.
\end{align*}

Using Equation \ref{eqn_this} note that $2n^3 - 13 n^2 + 30n -24 \leqslant E_{[2,2,1^{n-3}]}([n,1])-E_{[2,2,1^{n-4}]}([n-1,1])$ for all $n \ge 15$. \\

\noindent Case 2: Let $\lambda_i^+=\lambda_i+1$ for $i\geq 2$. Since $\lambda$ is a partition, we will have $i \lambda_i < \sum_{\ell = 1}^i \lambda_{\ell} \le n$. Moreover, since $2 \le i \le n-2$, then $\lambda_i \leqslant \frac{n}{2}$ and $(\lambda_i -\frac{i}{2}+\frac{1}{2})\leqslant \lambda_i$.  Lastly, if $\lambda \not \in \{[n], [n-1,1]\}$, then $p_1 \le  2n^2-9n+12$. Removing the negative terms from Equation \ref{equ:42} and substituting the maximum values of  $p_1$, $i$, $\lambda_i -\frac{i}{2}+\frac{1}{2}$, and $i\lambda_i$, we obtain
\begin{align*}
 \Delta E_{([2,2,1^{n-4}]}(\lambda)\leqslant& (\lambda_i)p_1+4i\lambda_i+\frac{1}{2}ni+\frac{3}{2}n\\
\leqslant&\left(\frac{n}{2}\right)( 2n^2-9n+12)+4n+\frac{1}{2}n(n-2)+\frac{3}{2}n  \\
\leqslant & n^3-4n^2+\frac{21n}{2}.
\end{align*}

From the inequality above and Equation \ref{eqn_this}, we conclude that, $\Delta E_{[2,2,1^{n-4}]}(\lambda)\leqslant n^3-4n^2+\frac{21n}{2} < E_{[2,2,1^{n-3}]}([n,1])-E_{[2,2,1^{n-4}]}([n-1,1])$ for all $n \ge 15$.  \\

\noindent Case 3: Let $\lambda^+_{k+1}=1$, that is, $i=k+1$ and $\lambda_k=0$. Substituting these two values into Equation \ref{equ:42}, we obtain 

\begin{equation} \label{equ:42k}
\Delta E_{[2,2,1^{n-4}]}(\lambda) = \left(\frac{-k}{2}\right)p_1-(k+1)^2 + \frac{1}{2}n(k+1)+(k+1)+\frac{3}{2}n. 
\end{equation}

Observe that $k \leqslant n$. By Statement 2 of Lemma \ref{lem:p1}, we have $-p_1\leqslant n^2-2n$. Removing the second term in Equation \ref{equ:42k} and noting that $(\frac{-k}{2})p_1 \leqslant \frac{n}{2} (n^2-2n)$, we have the following inequality:

\begin{align*}
\Delta E_{[2,2,1^{n-4}]}(\lambda) \leqslant &~ \left(\frac{-k}{2}\right)p_1 + \frac{1}{2}n(k+1)+(k+1)+\frac{3}{2}n\\
\leqslant &~ \frac{n}{2}(n^2-2n)+\frac{1}{2}(n)(n+1)+(n+1)+\frac{3}{2}n\\
\leqslant&~ \frac{n^3}{2}-\frac{n^2}{2}+3n+1.
\end{align*}

It follows that $\Delta E_{[2,2,1^{n-4}]}(\lambda)\leqslant  \frac{n^3}{2}-\frac{n^2}{2}+3n+1< E_{[2,2,1^{n-3}]}([n,1])-E_{[2,2,1^{n-4}]}([n-1,1])$ for $n \ge 15$. Therefore, the second largest eigenvalue of $A_{[2,2,1^{n-4}]}$ is $\phi_{[2,2,1^{n-4}]}^{[n-1,1]}$. The formula for $\phi_{[2,2,1^{n-4}]}^{[n-1,1]}$ and the spectral gap can be easily verified using Lemma \ref{lem:degree} and Lemma \ref{lem:foreig}. \end{proof}

\begin{theorem} 
\label{thm:E_8}
For all $n \ge 6$, the second largest eigenvalue of $A_{[4,1^{n-4}]} $ is 
$$\phi^{[n-1,1]}_{[4,1^{n-4}]} = 2n^4-20n^3+66n^2-80n+24.$$
The spectral gap of the graph corresponding to $A_{[4,1^{n-4}]} $ is 
$$ \phi^{[n]}_{[4,1^{n-4}]} - \phi^{[n-1,1]}_{[4,1^{n-4}]} = 8n^3-44n^2+68n-24.$$
\end{theorem}

\begin{proof} We have already observed, after Conjecture \ref{conj:main}, that the claim about the second largest eigenvalue is valid for all $n \leq 15$. Consequently, Lemma \ref{induction_lemma} implies that the proof reduces to establishing
\begin{align*}
    \Delta E_{[4,1^{n-4}]}(\lambda)\leq E_{[4,1^{n-3}]}([n,1])-E_{[4,1^{n-4}]}([n-1,1])
\end{align*}
for all $n\geq 15$ and for every partition $\lambda$ of $n$ with $\lambda\neq [n]$. This time the base symmetric function provides 
\begin{align*}
\phi_{[4,1^{n-4}]}^{\lambda}= E_{[4,1^{n-4}]}(\lambda) =  \frac12 p_3 - \frac94 p_2 + \frac{(11-2t)}{2} p_1 + \frac{(8t^2-23t)}{8}. 
\end{align*}
We have 
\begin{align*}
\Delta(8t^2-23t) =& [8(2n+2)^2-23(2n+2)] - [8(2n)^2 - 23(2n)] \\
=& 64n-14, \\
\Delta((11-2t)p_1) =& [(11-2(2n+2))(p_1 + \Delta p_1)] - [(11-2(2n))p_1] \\ 
=& - 4 p_1 + (7-4n)(\Delta p_1) \\ 
=& -4p_1 + (28\lambda_i -14i +21 -16n\lambda_i +8ni -12n). 
\end{align*}
Therefore, by referring to Table \ref{Tab:FormulaeDelta} for the formula of $\Delta p_3$ and  $\Delta p_2$, we obtain that
\begin{align*}
\Delta E_{[4,1^{n-4}]}(\lambda) =& \frac12 \Delta p_3 - \frac94 \Delta p_2 + \frac12 \Delta((11-2t)p_1) + \frac18 \Delta(8t^2-23t) \\ 
=& \frac12 (-2i^3+12i^2\lambda_i -24i \lambda_i^2 + 16 \lambda_i^3 +9i^2 -36i\lambda_i +36\lambda_i^2 - 15i \\
 &+ 30\lambda_i+9)  - \frac94 ( 2i^2-6i+5-8i\lambda_i + 12 \lambda_i + 8 \lambda_i^2 )+ \frac12 ( -4p_1 + 28\lambda_i   \\
 & -14i  +21 -16n\lambda_i +8ni -12n)+ \frac18 (64n-14).  \\ 
\end{align*}
 \vspace{-1cm}
 
\noindent Simplifying the above equation gives

\begin{equation} \label{eq:E42c1}
\Delta E_{[4,1^{n-4}]}(\lambda) =8\lambda_i^3 -12i \lambda_i^2   +(6i^2-8n+2)\lambda_i -2p_1  -i^3 +4ni  - i   + 2n+2. 
\end{equation}

\noindent Case 1: Suppose $\lambda^+ = [\lambda_1+1,\lambda_2,\dots,\lambda_k]$, so $i=1$.  For the case $\lambda^+ = [n,1]$ we have $\lambda_1 = n-1$ and $p_1 = -1 + (2n-3)(n-1)$. Therefore, we have 
\begin{align*}
    E_{[4,1^{n-3}]}([n,1])-E_{[4,1^{n-4}]}([n-1,1])  = 8n^3-48n^2+80n-32.
\end{align*} 
  
We now show that $\Delta E_{[4,1^{n-4}]}(\lambda) \leqslant 8n^3-48n^2+80n-32$ for all $\lambda \not \in \{[n], [n-1,1]\}$. To do so, we consider two sub-cases. 

\noindent Sub-case 1.1: Let $\lambda_1=n-2$. Then $\lambda \in \{[n-2,2], [n-2, 1,1]\}$. Additionally, we point out that 

 \vspace{-0.5cm}
$$p_1([n-2, 2])=\frac{(2n-4)(2n-5)}{2}+2; ~~~~~p_1([n-2, 1,1])=\frac{(2n-4)(2n-5)}{2}-4.$$

\noindent We now evaluate Equation \ref{eq:E42c1} for each $\lambda \in \{[n-2,2], [n-2, 1,1]\}$ by letting $i=1$:  

 \vspace{-0.2cm}
\begin{align*}
\Delta E_{[4,1^{n-4}]}([n-2,2])&=8n^3-72n^2+192n-152;\\
\Delta E_{[4,1^{n-4}]}([n-2,1,1])&=8n^3-72n^2+192n-140.\\
 \end{align*}
 \vspace{-1cm}
  
\noindent Clearly, for both $\lambda \in \{[n-2,2], [n-2, 1,1,]\}$, $\Delta E_{[4,1^{[n-4]}}(\lambda)\leqslant 8n^3-48n^2+80n-32$ for all $n\geqslant 15$.\\

\noindent Sub-case 1.2: Let $\lambda_1\leqslant n-3$. We derive an upper-bound on $\Delta E_{[4,1^{n-4}]}(\lambda)$ as follows. By Statement 2 of Lemma \ref{lem:p1}, we have $-2p_1\leqslant 2n^2-4n$. Next, we note that $i=1$ and $\lambda_1\leqslant n-3$. Removing all negative terms from Equation \ref{eq:E42c1} and substituting the aforementioned upper-bound gives the following inequality:
\begin{align*}
\Delta E_{[4,1^{n-4}]}(\lambda)\leqslant& 8\lambda_i^3 +(6i^2+2)\lambda_i-2p_1+4ni+2n+2\\
&\leqslant 8(n-3)^3+8(n-3)+(2n^2-4n)+6n+2\\
&\leqslant 8n^3-70n^2+226n-238.
\end{align*}
   \vspace{-0.5cm}  
   
\noindent We see that   $\Delta E_{[4,1^{n-4}]}(\lambda)\leqslant 8n^3-48n^2+80n-32$ for all $n\geqslant 15$ and thus, the claim follows. 
  
\noindent  Case 2: Let $\lambda_i^+=\lambda_i+1$ for $2\leqslant i \leqslant n$. We have the following inequalities: $2 \le i \le k \le n$, $2 \le i \lambda_i \le n-1$, and $1 \le \lambda_i < \frac{n}{2}$.  If $\lambda \not\in \{[n], [n-1,1]\}$, Statement 2 of Lemma \ref{lem:p1} implies that $-2p_1\leqslant 2n^2-4n$.  Furthermore we note that $|6i^2-8n+2|\leqslant 6i^2+8n+2$. From Equation \ref{eq:E42c1}, this means that
\begin{align*}
\Delta E_{[4,1^{n-4}]}(\lambda) &\leqslant 8\lambda_i^3 +(6i^2+8n+2)\lambda_i+(2n^2-4n)+4ni+ 2n+2\\
& \leqslant ~8\left(\frac{n}{2}\right)^3 +(6n^2+8n+2)\left(\frac{n}{2}\right)+ 6n^2-2n+2\\
& \leqslant~ 4n^3+10n^2-n+2.
\end{align*}
   
\noindent Since $4n^3+10n^2-n+2 \leqslant E_{[4,1^{n-3}]}([n,1])-E_{[4,1^{n-4}]}([n-1,1])$ for all $n \ge 15$, the claim follows. \\
 
\noindent Case 3: Let $\lambda^+_{k+1}=1$, that is, $i=k+1$ and $\lambda_k=0$. From Equation \ref{eq:E42c1}, this means that
\begin{align} 
\Delta E_{[4,1^{n-4}]}(\lambda) &= -(k+1)^3 -(k+1) -2 p_1 +4n(k+1) +2n+2 \nonumber \\
& = -k^3 -3k^2 +(4n-4)k  -2p_1 +6n. \label{eq:E42k}
\end{align}

\noindent By removing all negative terms from Equation \ref{eq:E42k} and noting that $k\leqslant n$ and $-2p_1\leqslant 2n^2-4n$, we have that
\begin{align} 
\Delta E_{[4,1^{n-4}]}(\lambda) \leqslant& 4n^2+ (2n^2-4n)+6n  \nonumber \\ 
\leqslant& 6n^2 +2n.  \label{eq:E42k1} 
\end{align}

Inequality \ref{eq:E42k1} implies that $\Delta E_{[4,1^{n-4}]}(\lambda)  \leqslant E_{[4,1^{n-3}]}([n,1])-E_{[4,1^{n-4}]}([n-1,1])$  for all $n \ge 15$. Hence, the statement holds in Case 3 as well. Therefore, we have completed the proof of the first statement. The second largest eigenvalue and spectral gap can then be computed using Lemma \ref{lem:degree} and Lemma \ref{lem:foreig}. \end{proof}

We point out that the spectral gap for $A_{[4,1^{n-4}]} $ is precisely four times the spectral gap for $A_{[2,2,1^{n-4}]}$. In Section \ref{S:gap} below, we generalize this observation by deriving a combinatorial formulae that computes the difference between $\phi^{[n]}_{\mu}$ and $\phi^{[n-1,1]}_{\mu}$ for certain partitions $\mu$.

\begin{theorem} 
\label{thm:E_32}
For all $n \ge 7$, the second largest eigenvalue of $A_{[3,2,1^{n-5}]} $ is 
$$\phi^{[n-1,1]}_{[3,2,1^{n-5}]} = \frac{2}{3}(2n^5-30n^4+165n^3-405n^2+418n-120).$$
The spectral gap of the graph corresponding to $A_{[3,2,1^{n-5}]} $ is 
$$ \phi^{[n]}_{[3,2,1^{n-5}]} - \phi^{[n-1,1]}_{[3,2,1^{n-4}]} = \frac{1}{3}(20n^4-190n^3+610n^2-740n+240).$$
\end{theorem}

\begin{proof} As already noted, after Conjecture \ref{conj:main} that the statement holds for all $n \leq 15$. Therefore, by Lemma \ref{induction_lemma}, to prove the statement it suffices to show that 
$$
\Delta E_{[3,2,1^{\,n-5}]}(\lambda) 
\;\leq\; 
E_{[3,2,1^{\,n-4}]}([n,1]) - E_{[3,2,1^{\,n-5}]}([n-1,1])
$$
for all $n\geq 15$ and for every partition $\lambda$ of $n$ with $\lambda \neq [n]$. Recall that, the base symmetric function provides 
\begin{align*}
  \phi^\lambda_{[3,2,1^{n-5}]}=& E_{[3,2,1^{-5}]}(\lambda)\\
  =& -2p_3+\frac{1}{4}p_1p_2+\frac{1}{8}(60-t)p_2-\frac{1}{2}p_1^2+\frac{1}{8}(29t-120-t^2)p_1+\\
  &\frac{1}{16}(116t-47t^2+t^3).  
\end{align*}

Refer to Table \ref{Tab:FormulaeDelta} for $\Delta p_3$. $\Delta p_2$, and $\Delta p_1$. Note that
\begin{align*}
\Delta p_1p_2&=32\lambda_i^3+(72-48i)\lambda_i^2+(24i^2-72i+56)\lambda_i-4i^3+18i^2-28i\\
&+15+(4\lambda_i-2i+3)p_2+(8\lambda_i^2+(12-8i)\lambda_i+2i^2-6i+5)p_1;\\
\Delta ((60-t)p_2)&=(58-2n)(2i^2-6i+5-8i\lambda_i+12\lambda_i+8\lambda_i^2)-2p_2; \\
\Delta((29t-120-t^2)p_1)&=(-4n^2+50n-66)(4\lambda_i-2i+3)+(54-8n)p_1; \\
\Delta (116t-47t^2+t^3)&=24n^2-352n+52.  
\end{align*}
We then have
 \begin{align*}
\Delta E_{[3,2,1^{n-5}]}(\lambda)=-2\Delta p_3+\frac{1}{4} \Delta p_1p_2+ \frac{1}{8}\Delta((60-t)p_2)-\frac{1}{2}\Delta p_1^2+\\ \frac{1}{8}\Delta((29t-120-t^2)p_1)+\frac{1}{16}\Delta(116t-47t^2+t^3).
\end{align*}
With the assistance of a computer, we obtain

\begin{multline}\label{eq:E32deltai}
\Delta E_{[3,2,1^{n-5}]}(\lambda)=-24\lambda_i^3+(-2n+36i-4)\lambda_i^2 +(-2n^2+2ni-18i^2+\\ 22n+4i-4)\lambda_i+n^2i-\frac{ni^2}{2}+3i^3-11ni-i^2-\frac{9}{2}n+2i-4+ \frac{1}{2}(2\lambda_i-i+1)p_2+\\\left(2\lambda_i^2-(2i+1)\lambda_i+\frac{1}{2}(i^2+i+10)-n\right)p_1.
\end{multline}
 
\noindent Case 1: Let $\lambda_1^+=\lambda_1+1$, that is, $i=1$. Letting $i=1$ in Equation \ref{eq:E32deltai}, we see that
\begin{multline} \label{eq:E32deltac1}
\Delta E_{[3,2,1^{n-5}]}(\lambda)=-24\lambda_1^3+(32-2n)\lambda_1^2+(-2n^2+24n-18)\lambda_1+n^2-16n\\ +\lambda_1p_2+(2\lambda_1^2-3\lambda_1-n+6)p_1.
\end{multline}
 
 If $\lambda=[n-1,1]$, then $p_1=(n-1)(2n-3)-1$ and $p_2=\frac{(2n-3)(2n-2)(4n-5)}{6}+1$. Substituting these two values along with $\lambda_1=n-1$ into Equation \ref{eq:E32deltac1} implies that
\begin{align}\label{eqn_atb}
E_{[3,2,1^{n-4}]}([n,1])-E_{[3,2,1^{n-5}]}([n-1,1])=\frac{20}{3}n^4-\frac{200}{3}n^3+\frac{670}{3}n^2-\frac{850}{3}n+100.
\end{align} 
Next, we show that $\Delta E_{[3,2,1^{n-5}]}(\lambda)\leq \frac{20}{3}n^4-\frac{200}{3}n^3+\frac{670}{3}n^2-\frac{850}{3}n+100$ whenever $\lambda \not \in \{[n-1,1], [n]\}$ and $n\geq 15$. To do so, we consider two sub-cases. \\
 
 \noindent Sub-case 1.1: Let $\lambda \in \{ [n-2,2], [n-2,1,1]\}$. For both partitions of interest, we have that $\lambda_1=n-2$. First, we note that 
 \begin{align*}
p_1(c([n-2,2])) &=\frac{(2n-4)(2n-5)}{2}+2; \\ 
p_1(c([n-2,1,1])) &=\frac{(2n-4)(2n-5)}{2}-4; \\ 
p_2(c([n-2,2]))&=\frac{(2n-5)(2n-4)(4n-9)}{6}+6; \\ 
p_2(c([n-2,1,1]))&=\frac{(2n-5)(2n-4)(4n-9)}{6}+6.
\end{align*}
Substituting these into Equation \ref{eq:E32deltac1} implies that
\begin{align*}
\Delta E_{[3,2,1^{n-5}]}([n-2,2])&=\frac{20}{3}n^4-\frac{280}{3}n^3+\frac{1384}{3}n^2-\frac{2804}{3}n+644;\\
\Delta E_{[3,2,1^{n-5}]}([n-2,1,1])&=\frac{20}{3}n^4-\frac{280}{3}n^3+\frac{1348}{3}n^2-\frac{2588}{3}n+524.
\end{align*}

Using this, it is not too difficult to verify that $\Delta E_{[3,2,1^{n-5}]}([n-2,2]) \leq\frac{20}{3}n^4-\frac{200}{3}n^3+\frac{670}{3}n^2-\frac{850}{3}n+100$ and  $\Delta E_{[3,2,1^{n-5}]}([n-2,1,1]) \leq\frac{20}{3}n^4-\frac{200}{3}n^3+\frac{670}{3}n^2-\frac{850}{3}n+100$ for all $n\geq 15$.\\
 
\noindent Sub-case 1.2: Let $\lambda_1 \leqslant n-3$. Note that, in this case, $p_1\leqslant \frac{(2n-4)(2n-5)}{2}+2$ for all partitions with $\lambda_1 \leqslant n-3$. Furthermore, we note that, if $1 \leqslant \lambda_1 \leqslant n-3$, then the maximum of $2\lambda_1^2-3\lambda_1$ occurs when $\lambda_1=n-3$. Therefore, we have
\begin{align*}
2\lambda_1^2-3\lambda_1-n+6\leqslant 2(n-3)^2-3(n-3)+6-n=2 n^2 - 16 n + 33.
\end{align*}
Therefore, 
\begin{align*}
(2\lambda_1^2-3\lambda_1-n+6)p_1\leqslant (2 n^2 - 16 n + 33)\left(\frac{(2n-4)(2n-5)}{2}+2\right).    
\end{align*}
Furthermore, we see that $p_2\leqslant \frac{(2n-5)(2n-4)(4n-9)}{6}+6$ when $\lambda \not \in \{[n-1,1], [n]\}$. Removing negative terms in Equation \ref{eq:E32deltac1} and using the inequalities on $p_2$ and $(2\lambda_1^2-3\lambda_1-n+6)p_1$, we have
\begin{multline}\label{eq:E32deltabound}
\Delta E_{[3,2,1^{n-5}]}(\lambda)\leqslant 32\lambda_1^2 +24n \lambda_1+n^2+\lambda_1\left(\frac{(2n-5)(2n-4)(4n-9)}{6}+6\right)\\ + (2 n^2 - 16 n + 33)\left(\frac{(2n-4)(2n-5)}{2}+2\right).
\end{multline}
 
\noindent Substituting $\lambda_1\leq n-3$ in Inequality \ref{eq:E32deltabound} implies that
\begin{align*}
   \Delta E_{[3,2,1^{n-5}]}(\lambda)\leqslant \frac{20}{3}n^4-76n^3+\frac{1156}{3}n^2-898n+756. 
\end{align*}
Using this, it is not too difficult to verify that $\Delta E_{[3,2,1^{n-5}]}(\lambda) \leq\frac{20}{3}n^4-\frac{200}{3}n^3+\frac{670}{3}n^2-\frac{850}{3}n+100$ for all $n\geq 15$. Thus, the claim follows for this case. 
 
\noindent Case 2: Let $\lambda_i^+=\lambda_i+1$ for $2 \leqslant i \leqslant k$. Since $i \geqslant 2$, we have $\lambda_i \leqslant \frac{n}{2}$. We derive an upper-bound on $\Delta E_{[3,2,1^{n-5}]}(\lambda)$ by bounding the individual terms in Equation \ref{eq:E32deltai}. To begin with, by removing all the negative terms from Equation \ref{eq:E32deltai}, we obtain 
\begin{align}\label{eqn_obt}
\Delta E_{[3,2,1^{n-5}]}(\lambda)&\leqslant 36i\lambda_i^2+(2ni+22n+4i)\lambda_i+ n^2i+3i^3+2i \nonumber\\ 
&+\frac{1}{2}(2\lambda_i-i+1)p_2+\left(2\lambda_i^2-(2i+1)\lambda_i+\frac{1}{2}(i^2+i+10)-n\right)p_1.
\end{align}

\noindent If $\lambda \not\in\{ [n-1,1], [n]\}$, then $|p_1|\leqslant \frac{(2n-5)(2n-4)}{2}+2$. Furthermore, we note that, $2\leq i \leqslant n-2$, $i\lambda_i\leq n$, and $\lambda_i\leqslant \frac{n}{2}$, therefore 
\begin{align*}
\left|2\lambda_i^2-(2i+1)\lambda_i+\frac{1}{2}(i^2+i+10)-n\right|&\leqslant 2\lambda_i^{2}+(2i+1)\lambda_i+\frac{1}{2}(i^2+i+10)+n\\
&\leqslant n^2+2n+6,
\end{align*}
and thus, we have that 
\begin{align*}
\left(2\lambda_i^2-(2i+1)\lambda_i+\frac{1}{2}(i^2+i+10)-n\right)p_1\leqslant (n^2+2n+6)\left(\frac{(2n-5)(2n-4)}{2}+2\right).
\end{align*}

Next, we note that, if $\lambda \not\in\{ [n-1,1], [n]\}$, then $p_2\leqslant \frac{(2n-5)(2n-4)(4n-9)}{6}+6$. Since $p_2>0$ and $i\geqslant 2$, then $(2\lambda_i-i+1)p_2\leqslant (2\lambda_i-1)p_2$. With the assistance of a computer, we can now further refine the upper-bound on $\Delta E_{[3,2,1^{n-5}]}$ from Equation \ref{eqn_obt} as follows
\begin{align*} 
\Delta E_{[3,2,1^{n-5}]}(\lambda)&\leqslant (36i\lambda_i)\lambda_i+2ni\lambda_i+22n\lambda_i+4i\lambda_i+ n^2i+3i^3+2i\\ 
&\ \ \ + \frac{1}{2}(2\lambda_i-1)\left( \frac{(2n-5)(2n-4)(4n-9)}{6}+6\right)\\
&\ \ \ +(n^2+2n+6)\left(\frac{(2n-5)(2n-4)}{2}+2\right)\\
&\leq \frac{10}{3}n^4-\frac{34}{3}n^3+\frac{277}{6}n^2-\frac{121}{6}n+56.
\end{align*}

Using the expression for $E_{[3,2,1^{n-4}]}([n,1])-E_{[3,2,1^{n-5}]}([n-1,1])$ from Equation \ref{eqn_atb}, we note that  $\Delta E_{[3,2,1^{n-5}]}(\lambda)\leqslant E_{[3,2,1^{n-4}]}([n,1])-E_{[3,2,1^{n-5}]}([n-1,1])$ for all $n \geqslant 15$, and thus the claim follows. \\

\noindent Case 3: Let $\lambda^+_{k+1}=1$, that is, $i=k+1$ and $\lambda_k=0$. In Equation \ref{eq:E32deltai}, we let $\lambda_{k}=0$ and $i=k+1$. This gives
\begin{multline}\label{eq:E32k}
\Delta E_{[3,2,1^{n-5}]}(\lambda)=3k^3+\left(8-\frac{n}{2}\right)k^2+(n^2-12n+9)k+n^2-16n\\-\frac{k}{2}p_2+\frac{1}{4}(2k^2+6k+24-4n)p_1.
\end{multline}

Since $k \leqslant n$, this means that $|2k^2+6k+24-4n|\leqslant 2k^2+6k+4n+24\leqslant 2n^2+10n+24$. Furthermore, we note that $|p_1| \leqslant \frac{(2n-5)(2n-4)}{2}+2$. Consequently 

\begin{multline}\label{eq:E32p1}
\frac{1}{4}(2k^2+6k+24-4n)p_1\leqslant \frac{1}{4}(2n^2+10n+24)\left(\frac{(2n-5)(2n-4)}{2}+2\right).
\end{multline}

\noindent Moreover, we note that, if $n \geqslant 12$, then $(n^2-12n+9)>0$. We now remove negative terms from Equation \ref{eq:E32k} and substitute Inequality \ref{eq:E32p1} together with $k \leqslant n$ to obtain
\begin{align*}
\Delta E_{[3,2,1^{n-5}]}(\lambda)\leqslant&~~ 3k^3+8k^2+(n^2-12n+9)k+n^2+\\ 
&~~\frac{1}{4}(2n^2+10n+24)\left(\frac{(2n-5)(2n-4)}{2}+2\right)\\
\leqslant&~~ n^4+\frac{9}{2}n^3-\frac{15}{2}n^2-15n+72.
\end{align*}

Lastly, from Equation \ref{eqn_atb}, we see that $ \Delta E_{[3,2,1^{n-5}]}(\lambda)\leq E_{[3,2,1^{n-4}]}([n,1])-E_{[3,2,1^{n-5}]}([n-1,1])$ for all $n \geqslant 15$, and thus, the claim follows. Therefore, we have completed the proof of the first statement. The second largest eigenvalue and the spectral gap can then be computed using Lemma \ref{lem:degree} and Lemma \ref{lem:foreig}.  \end{proof}

\begin{theorem} 
\label{thm:E_5}
For all $n \ge 6$, the second largest eigenvalue of $A_{[5,1^{n-5}]} $ is 
$$\phi^{[n-1,1]}_{[5,1^{n-5}]} = \frac{8}{5}(2n^5-30n^4+165n^3-405n^2+418n-120).$$
The spectral gap of the graph corresponding to $A_{[5,1^{n-5}]} $ is 
$$ \phi^{[n]}_{[5,1^{n-5}]} - \phi^{[n-1,1]}_{[5,5^{n-5}]} = 16n^4-152n^3+488n^2-592n+192.$$
\end{theorem}
\begin{proof} We have already observed, after Conjecture \ref{conj:main}, that the claim about the second largest eigenvalue is valid for all $n \leq 15$. Consequently, Lemma \ref{induction_lemma} implies that the proof reduces to establishing
\begin{align*}
    \Delta E_{[5,1^{n-5}]}(\lambda)\leq E_{[5,1^{n-4}]}([n,1])-E_{[5,1^{n-5}]}([n-1,1])
\end{align*}
for all $n\geq 15$ and for every partition $\lambda$ of $n$ with $\lambda\neq [n]$.
Recall that {the symmetric function provides}
\begin{align*}
\phi^\lambda_{[5,1^{n-5}]} &= E_{[5,1^{n-5}]}(\lambda)\\
&=\frac{1}{2}p_4 -4p_3+\frac{1}{2}(40-3t)p_2-p_1^2+(7t-34)p_1+\frac{1}{12}(217t-96t^2+5t^3).   \end{align*}
Refer to Table \ref{Tab:FormulaeDelta} for $\Delta p_2$ and $\Delta p_1$, and thus, we have that
\begin{align*}
\Delta p_4&=(-(i-1)+2\lambda_i)^4+(-(i-1)+2\lambda_i+1)^4;\\ 
\Delta ((40-3t)p_2)&=(34-6n)(2i^2-6i+5-8i\lambda_i+12\lambda_i+8\lambda_i^2)-6p_2; \\
\Delta ((7t-34)p_1)&=(14n-20)(4\lambda_i-2i+3)+14p_1; \\
\Delta (217t-96t^2+5t^3)&=120n^2-648n+90.  
\end{align*}
Refer to Table \ref{Tab:FormulaeDelta} for $\Delta p_1^2$ and $\Delta p_3$. We then evaluate:
\begin{align*}
\Delta E_{[5,1^{n-5}]}(\lambda)=&\frac{1}{2}\Delta p_4 -4\Delta p_3+\frac{1}{2}\Delta((40-3t)p_2)-\Delta p_1^2+
\Delta((7t-34)p_1)\\
&+\frac{1}{12}\Delta (217t-96t^2+5t^3).
\end{align*}
With the assistance of a computer, we obtain
\begin{multline}\label{eq:E5deltai}
\Delta E_{[5,1^{n-5}]}(\lambda)=16\lambda_i^4-(16+32i)\lambda_i^3+(24i^2+24i-24n+36)\lambda_i^2\\+(-8i^3+24ni+20n-12i^2-36i+16)\lambda_i+(8-8\lambda_i+4i)p_1-3p_2+\\i^4+2i^3+(9-6n)i^2+(-10n-8)i+10n^2-27n- 4.
\end{multline}

\noindent Case 1: Let $\lambda_1^+=\lambda_1+1$, that is, $i=1$. Letting $i=1$ in Equation \ref{eq:E5deltai}, we get
\begin{multline}\label{eq:E5delta}
\Delta E_{[5,1^{n-5}]}(\lambda)=16\lambda_1^4-48\lambda_1^3+(84-24n)\lambda_1^2+(44n-40)\lambda_1 \\ +(12-8\lambda_1)p_1-3p_2+10n^2-43n.
\end{multline}
 
If $\lambda=[n-1,1]$, then $p_1=\frac{(2n-2)(2n-3)}{2}-1$ and $p_2=\frac{(2n-3)(2n-2)(4n-5)}{6}+1$. Substituting these two values along with $\lambda_1=n-1$ into Equation \ref{eq:E5delta}
\begin{align}\label{eqn_gbt}
   E_{[5,1^{n-4}]}([n,1])-E_{[5,1^{n-5}]}([n-1,1])=16n^4-160n^3+536n^2-680n+240. 
\end{align} 
Our next step is to show that $\Delta E_{[5,1^{n-5}]}(\lambda)\leq E_{[5,1^{n-4}]}([n,1])-E_{[5,1^{n-5}]}([n-1,1])$ whenever $\lambda \not \in \{[n-1,1], [n]\}$ and $n\geq 15$. To do so, we consider two sub-cases: $\lambda_1\leq \frac{n}{2}$ and $\lambda_1> \frac{n}{2}$.\\
 
\noindent Sub-case 1.1: Let $ \lambda_1 \leqslant \frac{n}{2}$. Since $\lambda \not \in \{[n-1,1], [n]\}$, then $|p_1|\leqslant \frac{(2n-4)(2n-5)}{2}+2$. Furthermore, we note that $|12-8\lambda_1|\leqslant 12+4n$. This means that  
\begin{align*}
(12-8\lambda_1)p_1\leqslant (12+4n)\left(\frac{(2n-4)(2n-5)}{2}+2\right).
\end{align*} 
\noindent  Removing the negative terms from Equation \ref{eq:E5delta} and substituting the upper-bound for $(12-8\lambda_1)p_1$ into Equation \ref{eq:E5delta}, we obtain the following upper-bound:
\begin{align}\label{eq:E5final1}
\Delta E_{[5,1^{n-5}]}(\lambda)\leqslant16\lambda_1^4+84\lambda_1^2+ 44n\lambda_1+
(12+4n)\left(\frac{(2n-4)(2n-5)}{2}+2\right)+10n^2.
\end{align}
Next, we substitute $\lambda_1 \leq\frac{n}{2}$ into Inequality \ref{eq:E5final1} to obtain the following inequality:
\begin{align*}
\Delta E_{[5,1^{n-5}]}(\lambda)\leq n^4+8n^3+41n^2-60n+144.    
\end{align*} 
Using this and Equation \ref{eqn_gbt}, one can verify that $\Delta E_{[5,1^{n-5}]}(\lambda)\leq E_{[5,1^{n-4}]}([n,1])-E_{[5,1^{n-5}]}([n-1,1])$ for all $n\geq 15$.

\noindent Sub-case 1.2: Let $\lambda_1>\frac{n}{2}$. Lemma \ref{lem:p1} implies that $p_1(\lambda)>0$. This means that $(12-8\lambda_1)p_1(\lambda)<0$. We proceed by removing the terms involving $p_1$ and $p_2$ from Inequality \ref{eq:E5delta} to obtain:
\begin{equation}\label{eq:E5final2}
\Delta E_{[5,1^{n-5}]}(\lambda) \leqslant 16\lambda_1^4-48\lambda_1^3+(84-24n)\lambda_1^2+(44n-40)\lambda_1 +10n^2-43n.
 \end{equation}
 
If $\lambda \in \{[n-2,1,1], [n-2,2]\}$, then substituting $\lambda_1= n-2$ into Inequality  \ref{eq:E5final2} implies:
\begin{align*}
\Delta E_{[5,1^{n-5}]}(\lambda) \leqslant 16n^4-200n^3+906n^2-1691n+1056.
\end{align*} 
\noindent Since $16n^4-200n^3+906n^2-1691n+1056\leqslant E_{[5,1^{n-4}]}([n,1])-E_{[5,1^{n-5}]}([n-1,1])$ for all $n \geqslant 15$, the statement follows for $\lambda_1=n-2$. 
 
If $\lambda_1 \leqslant n-3$, we remove all the negative terms from Inequality \ref{eq:E5final2}. Noting that $(84-24n)<0$. Therefore, 
\begin{align*}
\Delta E_{[5,1^{n-5}]}(\lambda) \leqslant 16\lambda_1^4+44n\lambda_1+10n^2.
\end{align*}
We can substitute $\lambda_1 \leq n-3$ in the above inequality, thereby yielding
\begin{align*}
\Delta E_{[5,1^{n-5}]}(\lambda) \leqslant 16n^4-192n^3+918n^2-1860n+1296.
\end{align*}

Consequently, using this and Equation \ref{eqn_gbt}, one can verify that $\Delta E_{[5,1^{n-5}]}(\lambda)\leq E_{[5,1^{n-4}]}([n,1])-E_{[5,1^{n-5}]}([n-1,1])$ for all $n\geq 15$, and thus the claim follows.
 
\noindent Case 2: Let $\lambda_i^+=\lambda_i+1$ for $i\geqslant 2$. First, we remove all the negative terms from Equation \ref{eq:E5deltai}. Then, we derive an upper-bound on $\Delta E_{[5,1^{n-5}]}(\lambda)$ by bounding the non-negative individual terms in Equation \ref{eq:E5deltai}
\begin{align*}
\Delta E_{[5,1^{n-5}]}(\lambda)\leqslant& 16\lambda_i^4+ (24i^2+24i+36)\lambda_i^2+ (24ni+20n+16)\lambda_i+\\ 
&(8-8\lambda_i+4i)p_1+i^4+2i^3+9i^2+10n^2.
\end{align*}

We note that $2\leq i \leq n-2$, $i\lambda_i\leq n$, $\lambda_i\leq \frac{n}{2}$, and $|p_1|\leq \frac{(2n-3)(2n-2)}{2}+2$. Therefore,

\begin{align*}
    (8-8\lambda_i+4i)p_1 \leq (8+4n+4i)\left(\frac{(2n-3)(2n-2)}{2}+2\right).
\end{align*}

We can then further refine and rearrange the above bound for $\Delta E_{[5,1^{n-5}]}(\lambda_1)$ as follows:

\begin{align} \label{eq:inequE5}
\Delta E_{[5,1^{n-5}]}(\lambda)\leqslant& 16\lambda_i^4+ 24(i\lambda_i)^2+24(i\lambda_i)\lambda_i+36\lambda_i^2+ 24n(i\lambda_i)+(20n+16)\lambda_i+\nonumber \\ 
&(8+4n+4i)\left(\frac{(2n-3)(2n-2)}{2}+2\right)+i^4+2i^3+9i^2+10n^2.
\end{align}

Therefore, with the aid of a computer, we obtain
\begin{align*}
 \Delta E_{[5,1^{n-5}]}(\lambda)\leq 2n^4+10n^3+70n^2+4n+36.
\end{align*}

Using the previous inequality and Equation \ref{eqn_gbt}, we note that  $\Delta E_{[5,1^{n-5}]}(\lambda)\leqslant E_{[5,1^{n-4}]}([n,1])-E_{[5,1^{n-5}]}([n-1,1])$ for all $n \geqslant 15$, and thus the claim follows. 

\noindent Case 3: Letting $i=k+1$ and $\lambda_{k+1}=0$ in Equation \ref{eq:E5deltai}, we obtain
\begin{multline} \label{eq:E5k}
\Delta E_{[5,1^{n-5}]}(\lambda)=k^4+6k^3+(21-6n)k^2+(20-22n)k+10n^2\\-43n-3p_2+(12+4k)p_1.
\end{multline}

\noindent Removing all the negative terms and using the upper bound on $|p_1|$ in Inequality \ref{eq:E5k}, we have the following inequality:
\vspace{-3mm}
\begin{multline} \label{eq:inequE53}
\Delta E_{[5,1^{n-5}]}(\lambda)\leqslant k^4+6k^3+21k^2+20k+10n^2\\+(12+4k)\left(\frac{(2n-3)(2n-2)}{2}+2\right).
\end{multline}
Substituting $k\leq n$ into Inequality \ref{eq:inequE53} implies that
\begin{align*}
    \Delta E_{[5,1^{n-5}]}(\lambda)\leqslant n^4+14n^3+35n^2-20n+60.
\end{align*}

\noindent In conclusion, by using Equation \ref{eqn_gbt} we see that $\Delta E_{[5,1^{n-5}]}(\lambda)< E_{[5,1^{n-4}]}([n,1]) - E_{[5,1^{n-5}]}([n-1,1])$ for $n \geqslant 15$, and thus, the claim follows. Therefore, we have completed the proof of the first statement. The second largest eigenvalue and the spectral gap can then be computed using Lemma \ref{lem:degree} and Lemma \ref{lem:foreig}. \end{proof}

\section{Relationship between conjectured spectral gaps of certain matrices}\label{S:gap}

As stated in the introduction of this paper, the spectral gap of a graph is of particular interest.  The aim of this section is to further investigate the spectral gap of certain graphs in $\mathcal{A}_{2n}$. We write $g_{\mu}$ to denote the spectral gap of $X_{\mu}$. Since all eigenvalues of $X_{\mu}$ are integers, it follows that $g_{\mu}$ is itself an integer. If $\mu$ and $\mu'$ are two partitions of $2n$, then it is clear that there exists a $C \in \mathds{Q}$ such that $g_{\mu}=Cg_{\mu'}$. If $\mu'$ is obtained from $\mu$ by merging two parts of order at least two in an admissible way, we can derive a combinatorial expression for $C$ in terms of the sizes of the two parts being merged and their multiplicities. Therefore, we can generalize the observation made below the proof of Theorem \ref{thm:E_8} when $\mu$ is a partition with at least one part of size 1. Lastly, we compute the conjectured spectral gap of all associates in $\mathcal{A}_{2n}$ indexed by hook partitions by strategically constructing an equitable partition of $X_\mu$.   

Let $P$ be a perfect matching of $K_{2n}$ and $\sigma\in S_{2n}$. Recall that  $\sigma P$ is the perfect matching obtained by applying $\sigma$ to $V(K_{2n})$. If the edge $\{a,b\} \in P$, then the edge $\{\sigma(a), \sigma(b) \} \in \sigma P$. Lastly, the set of perfect matchings of $K_{2n}$ containing the edge $\{1,2\}$ is denoted $PM_{1,2}$.

We will start with the following two partitions of $n$:
$$
\mu = [\mu_1, \mu_2, \dots,\mu_i, \dots, \mu_j, \ldots,  \mu_\ell];~~~
\mu' = [\mu_1, \mu_2, \dots,  \mu_i + \mu_j, \ldots, \mu_\ell].
$$

Observe that $\mu'$ is obtained by merging two parts of $\mu$ and thus $\mu' \trianglerighteq \mu$. Let $n_i$ denote the number of parts in $\mu$ equal to $\mu_i$, $n_j$ denote the number of parts in $\mu$ equal to $\mu_j$, and $m$ denote the number of parts in $\mu'$ equal to $\mu_i + \mu_j$. Define
\begin{align}\label{eqn_Cij}
C_{i, j}(\mu)  = 
\begin{cases}
  \frac{ n_i (n_i-1) \mu_i}{  2  m  }       & \quad \textrm{if }   \mu_i = \mu_j;  \\ 
 \frac{n_i n_j \mu_i \mu_j }{m ( \mu_i + \mu_j) }    & \quad \textrm{otherwise.}
\end{cases}
\end{align}

In the proof of Lemma \ref{lem:DegreeRatio}, we use simple counting arguments to show that $v_{\mu'}=C_{i, j}(\mu) v_{\mu}$ when $\mu_i, \mu_j  > 1$. Our proof is structured as follows. Starting with a perfect matching $P$, we consider a specific neighbour of $P$ in $X_\mu$ which we call $Q$. It follows that $P \cup Q$ will contain a cycle of length $\mu_i$ and $\mu_j$. We then count the neighbours of $P$ in $X_{\mu'}$ that can be obtained from $Q$ by combining the $\mu_i$-cycle and $\mu_j$-cycle to form a $(\mu_i+\mu_j)$-cycle. This will  count each of the neighbours of $P$ in $X_{\mu'}$ $K$ times. We then compute the factor $K$ by which we overcount to derive $v_{\mu'}$ in terms of $v_{\mu}$. 

\begin{lemma}\label{lem:DegreeRatio}
Let $\mu = [\mu_1, \mu_2, \dots,\mu_i, \dots, \mu_j, \dots,  \mu_\ell]$,
and $\mu' = [\mu_1, \mu_2, \dots,\mu_i + \mu_j, \dots,  \mu_\ell]$ such that $\mu_i, \mu_j  > 1$. Then  $v_{\mu'} = C_{i, j}(\mu)  v_{\mu}$.
\end{lemma}

\begin{proof}
Let $P$ be a perfect matching in $PM_{1,2}(2n)$. We will describe a way to construct perfect matchings that are $\mu'$-related to $P$ from the perfect matchings that are $\mu$-related. Assume that $P$ and $Q$ are $\mu$-related, then the graph $P \cup Q$ has exactly $\ell$ even cycles of length $2\mu_1, 2\mu_2, \ldots,$ $2\mu_i, \dots, 2\mu_j, \dots,  2\mu_\ell$. We will consider two cases: $\mu_i \neq \mu_j$ and $\mu_i = \mu_j$. \\~

\noindent Case 1: Let $\mu_i \neq \mu_j$. We will first count the number of ways to construct a matching $Q'$ that is $\mu'$-related to $P$ from a $Q$ and then divide this count by the number of distinct perfect matchings $Q$  from which we can construct the perfect matching $Q'$.

We will assume that $Q$ includes the edge $\{a,b\}$ in a cycle of length $\mu_i$ in $P \cup Q$ and the edge $\{c,d \}$ in a cycle of length $\mu_j$ in $P \cup Q$. Let $Q' = \sigma Q$, where $\sigma \in S_{2n}$ such that $\sigma$ is one of four transpositions in the set $\{ (a,c),(a,d),(b,c),(b,d)\}$. The perfect matching $Q'$ has the same set of edges as $Q$ with the exception of  $\{a,b\}$ and $\{c,d\}$, which are replaced by either $\{a,c\}$ and $\{b,d\}$, or $\{a,d\}$ and $\{b,c\}$. See Figure \ref{fig:switch} for an example of this transformation. It follows that $P$ and  $Q'$ are $\mu'$ related since the $\mu_i$ and $\mu_j$ cycles from $P \cup Q$ are now merged to form a $(\mu_i + \mu_j)$-cycle in $P \cup Q'$.

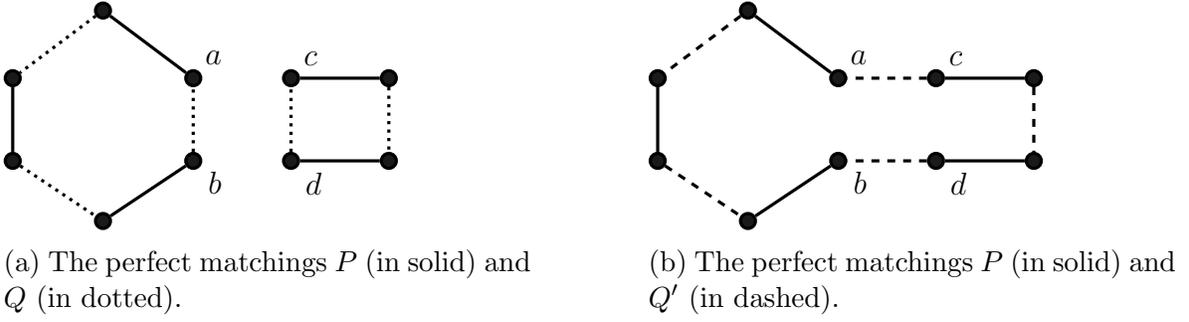
\begin{figure}[htpb]
\begin{subfigure}{0.45\textwidth}
\begin{tikzpicture}[
  very thick,
  every node/.style={circle,draw=black,fill=black!90}, inner sep=2]
 
  \node (n0) at (1.2,2.8) {};
  \node (n1) at (0.0,1.9) {};
  \node (n2) at (2.4,1.9)  [label=above right:$a$]{};
  \node (n3) at (0.0,0.8){};
  \node (n4) at (2.4,0.8)  [label=below right:$b$]{};
  \node (n5) at (1.2,0.0) {};
  \node (n6) at (3.7,1.9) [label= above right:$c$]{};
  \node (n7) at (5.0,1.9) {};
  \node (n8) at (3.7,0.8)  [label=below right:$d$]{};
  \node (n9) at (5.0,0.8) {};

   \path[every node/.style={font=\sffamily\small},  dotted]
   (n4) edge (n2)
   (n0) edge (n1)
   (n3) edge (n5)
      (n7) edge (n9)
   (n8) edge (n6);
   
  \path[every node/.style={font=\sffamily\small}]
   (n1) edge (n3)
   (n0) edge (n2)
   (n5) edge (n4)
   (n9) edge (n8)
   (n7) edge (n6);
 
\end{tikzpicture}
\caption{The perfect matchings $P$ (in solid) and $Q$ (in dotted). }
\end{subfigure}
\hfill
\begin{subfigure}{0.45\textwidth}
\begin{tikzpicture}[
  very thick,
  every node/.style={circle,draw=black,fill=black!90}, inner sep=2]
 
  \node (n0) at (1.2,2.8) {};
  \node (n1) at (0.0,1.9) {};
  \node (n2) at (2.4,1.9)  [label=above right:$a$]{};
  \node (n3) at (0.0,0.8){};
  \node (n4) at (2.4,0.8)  [label=below right:$b$]{};
  \node (n5) at (1.2,0.0) {};
  \node (n6) at (3.7,1.9) [label= above right:$c$]{};
  \node (n7) at (5.0,1.9) {};
  \node (n8) at (3.7,0.8)  [label=below right:$d$]{};
  \node (n9) at (5.0,0.8) {};

   \path[every node/.style={font=\sffamily\small},  dashed]
   (n6) edge (n2)
   (n0) edge (n1)
   (n3) edge (n5)
    (n7) edge (n9)
   (n8) edge (n4);
   
  \path[every node/.style={font=\sffamily\small}]
   (n1) edge (n3)
   (n0) edge (n2)
   (n5) edge (n4)
   (n9) edge (n8)
   (n7) edge (n6);
\end{tikzpicture}
\caption{The perfect matchings $P$ (in solid) and $Q'$ (in dashed).}
\end{subfigure}
\caption{Illustration of the merging of a $4$-cyle with a $6$-cycle by the action of $\sigma=(a,d)$ on perfect matching $Q$. }
\label{fig:switch}
\end{figure}

Next, we count the number of transpositions $\sigma$ such that $P$ and $\sigma Q$ are $\mu'$-related. There are $n_i n_j$ to select the $2\mu_i$-cycle and $2\mu_j$-cycle from $P \cup Q$ that will be merged to form a $2(\mu_i + \mu_j)$-cycle in $P \cup \sigma Q$. Additionally, there are $\mu_i$ ways to select the edge $\{a,b\}$ from the $2\mu_i$-cycle in $P \cup Q$, and $\mu_j$ ways to select the edge $\{ c, d \}$ from the $2\mu_i$-cycle in $P \cup Q$.  If $\sigma \in \{ (a,c), (a,d), (b,c), (b,d)\}$, the perfect matchings $P$ and $\sigma Q$ are $\mu'$-related. Therefore, there are $4 n_i n_j \mu_i \mu_j$ transpositions $\sigma$ such that $\sigma Q$ and $P$ are $\mu'$-related. Note that $(a,c) Q=(b,d)Q$ and $(a,d)Q=(b,c) Q$, and thus, the above count does not count the number of distinct matchings $Q'$ that arises from this transformation. We remedy this overcount below.  

Given a perfect matching $Q'$ that is $\mu'$-related to $P$, we now count the number of matchings $Q$ that are $\mu$-related to $P$ for which there exists a transposition $\sigma$ such that $Q' = \sigma Q$. We do so by counting the number of transpositions in $S_{2n}$ that transform a $2(\mu_i +\mu_j)$-cycle in $P \cup Q'$ into two cycles of lengths of $2\mu_i$ and  $2\mu_j$. Note that we assume that $\mu_j, \mu_j>1$. 

 There are $m$ ways to select the $2(\mu_i +\mu_j)$-cycle from $P \cup Q'$ that is to be broken. We must remove two edges from this cycle. There are $\mu_i + \mu_j$ ways to select an edge $\{a,c\}$ from $Q'$. The second edge of $Q'$ to be removed, call it $\{b,d\}$, must be at distance of $2\mu_i$ from the first edge and can be selected in one of two ways by moving clockwise or counter-clockwise in the $(\mu_i + \mu_j)$-cycle. Observe that $2\mu_i>2$.  Without loss, of generality, we can assume that $\dist\{a,d\} = \dist\{b,c \}$. This means that $\sigma \in \{ (a,d), (b,c) \}$. In conclusion,  $Q$ has the same edges as $Q'$ with the exception of  $\{a,c\}$ and $\{b,d\}$ which are replaced with $\{a, b\}$ and $\{c,d\}$. Consequently, matchings $P$ and $Q$ are $\mu$-related. Thus $Q'$ can be formed from exactly $ 4 m ( \mu_i + \mu_j )$ perfect matchings $Q$. 

Let $v_{\mu}$ is the number perfect matchings $Q$ that are $\mu$-related to $P$. Dividing the number of perfect matchings $Q'$ that are $\mu'$ related to $P$ by the number of different $Q$ that form a fixed $Q'$ gives rise to the following equality:

$$
v_{\mu'} = \frac{n_i n_j \mu_i \mu_j}{   m ( \mu_i + \mu_j) } v_{\mu}.
$$

\noindent Case 2: Let $\mu_i = \mu_j$. Since our computations are similar to those given in Case 1, we will omit these computations and only point out key differences. Clearly, we have $n_i = n_j$ and $\mu_i + \mu_j = 2 \mu_i$. In this case, there are $\binom{n_i}{2}$ ways of selecting the two $\mu_i$-cycles in $P \cup Q$ that are to be merged. 

Also different from Case 1 is the fact that, when breaking the $2(\mu_i +\mu_j)$-cycle in $P \cup Q'$, there is a unique way to select the second edge once the first edge has been selected. Consequently, $Q'$ can be formed from exactly $ 2 m( \mu_i + \mu_j )$ perfect matchings $Q$. In conclusion, we get the following expression
\begin{align*}
v_{\mu'} = \frac{4 \binom{n_i}{2} \mu_i \mu_j}{  2 m ( \mu_i + \mu_j) } v_{\mu}
=
\frac{  n_i (n_i-1) \mu_i}{  2  m  } v_{\mu}, 
\end{align*}
as claimed.
\end{proof}

We point out that we require $\mu_i, \mu_j  > 1$ because we must be able to break the $2(\mu_i+\mu_j)$-cycle into two cycles of length at least 4.

In Lemma \ref{lem:ratioSecond} below, we prove a statement similar to that of Lemma \ref{lem:DegreeRatio} for the eigenvalue occurring on the $[n-1, 1]$-eigenspace. To do so, we will first construct an equitable partition of $X_\mu$. This equitable partition yields the two eigenvalues of $X_{\mu}$ that occur on the eigenspaces indexed by $[n]$ and $[n-1,1]$. Although these eigenvalues are already known, the goal is to express $\phi^{[n-1, 1]}_\mu$ as a difference of two variables $a_\mu$ and $b_\mu$. We will then show that $a_{\mu'}=C_{i,j}(\mu)a_\mu$ and  $b_{\mu'}=C_{i,j}(\mu)b_\mu$, thereby yielding the desired relation between  $\phi^{[n-1, 1]}_\mu$  and  $\phi^{[n-1, 1]}_{\mu'}$ . 

First, we define an equitable partition of $X_\mu$ by defining the action of a subgroup of $S_{2n}$ such that the orbits of this action partitions the set of perfect matchings in precisely two orbits. Let $P \in \mathcal{M}_{2n}$ and $\sigma \in S_{2n}$. Recall that $S_{2n}$ acts transitively on the set of perfect matchings $\mathcal{M}_{2n}$. However, if we are to consider the action of the stabilizer of the edge $\{1,2\}$, which corresponds to the Young subgroup $S_2 \times S_{n-2}$, then this action has exactly two orbits. The first orbit is comprised of the matchings that contain the edge $\{1,2\}$ which is the set $PM_{1,2}$. The second orbit is the set $\mathcal{M}_{2n} \backslash PM_{1,2}(2n)$ which is the set of all matchings that do not contain the edge $\{1,2\}$. Note that $|PM_{1,2}(2n) | = |\mathcal{M}_{2n}|/(2n-1) = (2n-3)!!$. 

The orbits of the action of $S_2 \times S_{n-2}$ on $V(X_{\mu})$ form an equitable partition of $X_\mu$ and thus, $\pi = \{PM_{1,2}(2n),  \mathcal{M}_{2n} \backslash PM_{1,2}(2n)\}$ is an equitable partition of the graph $X_{\mu}$. Next, we construct the adjacency matrix of $X_\mu / \pi$, which will be a $2 \times 2$ matrix. Recall that $v_{\mu}$ is the degree of $X_\mu$. We let $a_\mu$ be the number of perfect matchings in $PM_{1,2}(2n)$ that are $\mu$-related to a fixed perfect matching in $PM_{1,2}(2n)$  and $b_\mu$ be the number of perfect matchings in $PM_{1,2}(2n)$ that are $\mu$-related to a fixed perfect matching in $\mathcal{M}_{2n} \backslash PM_{1,2}(2n)$. The adjacency matrix of $X_\mu / \pi$ is

\begin{equation} \label{equ:quotient}
Q_\mu(\pi) =
\begin{bmatrix}
a_\mu & v_{\mu}-a_\mu \\
b_\mu & v_{\mu}-b_\mu
\end{bmatrix}.
\end{equation}

 It is straightforward to compute the eigenvalues of $Q_\mu(\pi)$ which are also eigenvalues of $X_\mu$ by Lemma \ref{lem:quotient}. 
 
\begin{prop}\label{prop:evalueQuotient} \cite{GodMeaBook}
The eigenvalues of $Q_\mu$ are $v_{\mu}$ and $a_\mu-b_\mu$. The eigenvalue $v_{\mu}$ occurs on $[n]$-eigenspace, while $a_\mu-b_\mu$ occurs on the $[n-1, 1]$-eigenspace. 
\end{prop} 

Let $\tau_{\mu}$ be the eigenvalue of $X_\mu$ occurring on the $[n-1,1]$-eigenspace. Similarly to Lemma \ref{lem:DegreeRatio}, we can relate $\tau_{\mu}$ and $\tau_{\mu'}$. However,  in this case $\mu'$ must also have a part of size 1 to ensure that $a_\mu \neq 0$. Our proof of Lemma \ref{lem:ratioSecond} here differs from that of Lemma \ref{lem:DegreeRatio} as we make use of the eigenvalues of $X_\mu /\pi$. 

\begin{lemma} \label{lem:ratioSecond}
Let $\mu = [\mu_1, \mu_2, \dots,\mu_i, \dots, \mu_j, \dots,  \mu_\ell]$ 
and $\mu' = [\mu_1, \mu_2, \dots,\mu_i + \mu_j, \dots,  \mu_\ell]$, with $\mu_\ell =1$.
If $i, j <\ell$ and $\mu_i, \mu_j  > 1$
then  $\tau_{\mu'} =  C_{i, j}(\mu) \tau_{\mu}.$
\end{lemma}

\begin{proof}
 Let $a_{\mu}$ be the number of perfect matchings in $PM_{1,2}(2n)$ that are $\mu$-related to a fixed perfect matching in $PM_{1,2}(2n)$, and $a_{\mu'}$ the number of perfect matchings in $PM_{1,2}(2n)$ that are $\mu'$-related to a fixed perfect matching in $PM_{1,2}(2n)$. Since $\mu_\ell =1$, the perfect matching $P$ is adjacent to at least one perfect matching in $PM_{1,2}(2n)$. This means that $a_{\mu}, a_{\mu'} \neq 0$. We also use $b_{\mu} $ and $b_{\mu'}$ to denote the number of perfect matchings in $\mathcal{M}_{2n} \backslash PM_{1,2}(2n)$ that are $\mu$-related and $\mu'$-related, respectively, to a fixed perfect matching in $PM_{1,2}(2n)$.

Proposition~\ref{prop:evalueQuotient} implies that $\tau_\mu = a_{\mu} - b_{\mu} $ and  $\tau_{\mu'} = a_{\mu'} -b_{\mu'} $. Therefore, it suffices to show that $a_{\mu'} = C_{i, j}(\mu) a_{\mu}$ and $b_{\mu'} = C_{i, j}(\mu) b_{\mu}$. For ease of notation, we write $a=a_{\mu}$,  $a'=a_{\mu'}$,  $b=b_{\mu}$ and $b'=b_{\mu'}$.

Similarly to the proof of Lemma~\ref{lem:DegreeRatio}, we fix a perfect matching $P$. We assume that $P \in PM_{1,2}(2n)$ and let $Q$ be a neighbour of $P$ in $X_\mu$ such that $Q \in  PM_{1,2}(2n)$. We will count the number of perfect matchings $Q'\in PM_{1,2}(2n)$ such that $Q' = \sigma Q$, where $Q'$ is $\mu'$-related to $P$, and $ Q' \in PM_{1,2}(2n)$. As in  Lemma~\ref{lem:DegreeRatio}, we do this by counting the transposition $\sigma$ that merge the $\mu_i$-cycle with the $\mu_j$-cycle in $P \cup Q$. Observe that $P,Q \in PM_{1,2}(2n)$ and thus, both $P$ and $Q$ contain the edge $\{1,2\}$.
 
 As in the proof of Lemma \ref{lem:DegreeRatio}, we partition our proof in two cases.  
 
 \noindent Case 1: Let $\mu_i \neq \mu_j$. We begin by selecting a $2\mu_i$-cycle and $2 \mu_j$-cycle from $P \cup Q$; there are $n_i n_j$ ways to do so. Then, we select an edge from $Q$ from each cycle; there are $\mu_j\mu_j$ ways to do so. Lastly, there are $4$ transpositions $\sigma$ that will flip the edges so that $P$ and $\sigma Q$ are $\mu'$-related.
 Since $i,j < \ell$, the transposition $\sigma$ will fix the edge $\{1,2\}$. Therefore,  $Q' \in PM_{1,2}(2n)$.
 Just as in the proof of Lemma~\ref{lem:DegreeRatio}, for each $Q'$ that is $\mu'$-related to $P$, there are exactly $4 m ( \mu_i + \mu_j )$ pairs of transposition with a perfect matching  with $Q' = \sigma Q$. Therefore, we have that $a' = C_{i, j}(\mu) a$. 
 
Proof of $b' = C_{i, j}(\mu) b$ is identical except that $P \in \mathcal{M}_{2n}\backslash PM_{1,2}(2n)$ and $Q \in PM_{1,2}(2n)$.
 
 \noindent Case 2: Let $\mu_i = \mu_j$. The proof follows as in Case 2 of the proof of Lemma~\ref{lem:DegreeRatio}. \end{proof}

In conclusion, under the right hypothesis, we have a simple formulae relating the spectral gaps of $X_\mu$ and $X_{\mu'}$, which is described in Corollary \ref{cor:gap}. 

\begin{coro} \label{cor:gap}
Let $\mu = [\mu_1, \mu_2, \dots,\mu_i, \dots, \mu_j, \dots,  \mu_\ell]$ 
and $\mu' = [\mu_1, \mu_2, \dots,\mu_i + \mu_j, \dots,  \mu_\ell]$, with $\mu_\ell =1$,  $i, j <\ell$ and $\mu_i, \mu_j  > 1$. If $\phi_\mu^{[n-1,1]}$ and  $\phi
_{\mu'}^{[n-1,1]}$ are the second largest eigenvalue of $X_\mu$ and $X_{\mu'}$, respectively, then $$g_{\mu'}= C_{i, j}(\mu) g_\mu,$$ where $C_{i,j}(\mu)$ is as defined in Equation \ref{eqn_Cij}.
\end{coro}

Lastly, we can also compute the conjectured spectral gap of the associates indexed by hook partitions. This time, we use a quotient graph argument in conjunction with elementary counting arguments to obtain Proposition \ref{prop:gapdim}.

\begin{prop} \label{prop:gapdim}
If $\phi_{\mu}^{[n-1, 1]}$ is the second largest eigenvalue of $A_{ [n-\ell, 1^{\ell}]}$, then the spectral gap of $A_{ [n-\ell, 1^{\ell}]}$ is $$(2n-1) (2n-4) (2n-6) \cdots (2\ell+2).$$
\end{prop} 
\begin{proof}
The valency of  $A_{ [n-\ell, 1^\ell]}$  is $\binom{n}{\ell}(2n-2\ell -2)!!$, as per Lemma \ref{lem:degree}. Furthermore, we have  $a_\mu = \binom{n-1}{\ell-1}(2n-2\ell -2)!!$ and $b_\mu = \binom{n-2}{\ell}(2n-2\ell -4)!!$. Then the spectral gap is
\begin{align*}
v_\mu - a_\mu + b_\mu &= \left(\binom{n}{\ell} -  \binom{n-1}{\ell-1}\right)(2n-2\ell -2)!! + \binom{n-2}{\ell}(2n-2\ell -4)!!\\
&= \binom{n-1}{\ell}2^{n-\ell-1}(n-\ell-1)!+\binom{n-2}{\ell}2^{n-\ell-2}(n-\ell-2)!\\
&=\frac{2^{n-\ell-2}(n-2)!}{\ell!}(2(n-1)+1)\\
&=(2n-1)2^{n-\ell-2}(n-2)(n-3)\cdots(\ell+1)\\
&=(2n-1) (2n-4)(2n-6)\cdots (2\ell+2),  
\end{align*}
as claimed.
\end{proof}

\section{Future Work}\label{S:Fut}
In this paper, we studied the second largest eigenvalue (equivalently, the spectral gap) of certain graphs in the perfect matching association scheme. A natural question arising from this study is the following: among all graphs in the perfect matching association scheme, which relation $\mu$ gives rise to the graph with the smallest spectral gap? We propose the following conjecture.
\begin{conjecture} \label{conj:gap} Let $n\geq 5$. Then among the adjacency matrices of the perfect matching association scheme $\{A_{\mu}|\mu\vdash n, \mu \neq [1^{n}]\}$, the smallest spectral gap is attained by $A_{[2,1^{n-2}]}$.
\end{conjecture} 

The spectral gap and the diameter of a graph are intrinsically linked. As noted in Godsil and Royle \cite[Sec.~13.5]{GodRoy}, ``\emph{graphs with small values of $\lambda_2$} (the smallest nonzero Laplacian eigenvalue, which coincides with the spectral gap in regular graphs) \emph{tend to be elongated graphs of large diameter with bridges, whereas graphs with larger values of $\lambda_2$ tend to be rounder with smaller diameter, and larger girth and connectivity}". For a more precise mathematical formulation of this relationship, we refer the reader to \cite[Ex.~13.11]{GodRoy}, \cite{mohar1991eigenvalues}, and \cite{nilli1991second}.

Motivated by the above observation, we are also interested in the diameter of the graphs in $\mathcal{A}_{2n}$. In \cite{Jen}, Jennings showed that $X_{[2, 1^{n-2}]}$ has diameter $n-1$. Combining Proposition \ref{prop:gapdim} with Jennings' results on the diameter of $X_{[2,1^{n-2}]}$, we propose the following conjecture.
\begin{conjecture} \label{conj:diam}
The diameter of $X_{[2, 1^{n-2}]}$ is the largest diameter of all the graphs in the perfect matching association scheme.
\end{conjecture}

With the assistance of a computer, we are able to confirm Conjecture \ref{conj:gap} for $n\leq 15$ and Conjecture \ref{conj:diam} for all $n \leqslant 10$ by computing the diameter of all graphs in $\mathcal{A}_{2n}$. By comparing these results on the diameter of our graphs with the matrices of eigenvalues for $n \leqslant 15$, we observe the same pattern as  Godsil and Royle's \cite{GodRoy}. Lastly, preliminary observations also suggest that graphs that correspond to relations $\mu$ that are dominated by few partitions in the domination ordering have a small diameter while graphs indexed by partitions with many parts tend to have larger diameter.  We leave the problem of computing the diameter of the graphs in the perfect matching association scheme as an interesting open problem. 

Lastly, results of Section \ref{S:trace} imply that  the second largest eigenvalue in absolute value of most matrices occurs on the $[n-1,1]$-eigenspace. Since Equation \ref{eq:deg2} implies that the multiplicity of $\phi_\mu^\lambda$ cannot be large relative to $|\phi_\mu^\lambda|$,  it raises the question of which eigenspace can have the second largest eigenvalue, in absolute value, of $X_\mu$. The matrices of eigenvalues for small values of $n$ imply that the second largest eigenvalue, in absolute value, of most graphs occurs on the $[n-1,1]$-eigenspace with very few exceptions. It would be interesting to characterize these exceptions.

\newpage

\appendix
\section{Tables of Eigenvalues of perfect matching association scheme for small \texorpdfstring{$n$}{n}}\label{appendix:small_tables}

In this appendix, we provide the tables of eigenvalues of the perfect matching association scheme for $2\leq n\leq 7$ with the second largest eigenvalue in each column is highlighted in light grey color. Note that columns are indexed partitions of $n$, $\mu$, that in fact refer to even partitions  of $2n$, denoted $2\mu$ in this paper. 

\begin{table}[H]
    \centering
    \begin{minipage}{0.46\textwidth}
        \centering
        \begin{tabular}{|l|l|l|l|}
        \hline
            [$1^2$] & [2] & Rep. & Dim. \\ \hline
            1 & 2 & [4] & 1 \\ \hline
            1 & \cellcolor{lightgray}{-1} & [$2^2$] & 2 \\ \hline
        \end{tabular}
        \caption{Eigenvalues of $\mathcal{A}_{4}$}
        \label{tab_for_n=2}
    \end{minipage}%
    \hspace{0.05\textwidth} 
    \begin{minipage}{0.48\textwidth}
        \centering
        \begin{tabular}{|l|l|l|l|l|}
        \hline
            [$1^3$] & [2,1] & [3] & Rep. & Dim. \\ \hline
            1 & 6 & 8 & [6] & 1 \\ \hline
            1 & \cellcolor{lightgray}{1} & -2 & [4,2] & 9 \\ \hline
            1 & -3 & \cellcolor{lightgray}{2} & [$2^3$] & 5 \\ \hline
        \end{tabular}
        \caption{Eigenvalues of $\mathcal{A}_{6}$}
        \label{tab_for_n=3}
    \end{minipage}
\end{table}

\begin{table}[H]
    \centering
    \begin{tabular}{|l|l|l|l|l|l|l|}
    \hline
        [$1^4$] & [2,$1^2$] & [$2^2$] & [3,1] & [4] & Rep. & Dim. \\ \hline
        1 & 12 & 12 & 32 & 48 & [8] & 1 \\ \hline
        1 & \cellcolor{lightgray}{5} & -2 & {4} & -8 & [6,2] & 20 \\ \hline
        1 & 2 & \cellcolor{lightgray}{7} & -8 & -2 & [$4^2$] & 14 \\ \hline
        1 & -1 & -2 & -2 & \cellcolor{lightgray}{4} & [4,$2^2$] & 56 \\ \hline
        1 & -6 & 3 & \cellcolor{lightgray}{8} & -6 & [$2^4$] & 14 \\ \hline
    \end{tabular}
    \caption{Eigenvalues of $\mathcal{A}_{8}$}
    \label{tab_for_n=4}
\end{table}

\begin{table}[H]
    \centering
    \begin{tabular}{|l|l|l|l|l|l|l|l|l|}
    \hline
        [$1^5$] & [2,$1^3$] &  [$2^2$,1] & [3,$1^2$] &  [3,2] & [4,1] & [5] & Rep. & Dim. \\ \hline
        1 & 20 & 60 & 80 & 160 & 240 & 384 & [10] & 1 \\ \hline
        1 & \cellcolor{lightgray}{11} & 6 & \cellcolor{lightgray}{26} & -20 & \cellcolor{lightgray}{24} & -48 & [8,2] & 35 \\ \hline
        1 & 6 & 11 & -4 & \cellcolor{lightgray}{20} & -26 & -8 & [6,4] & 90 \\ \hline
        1 & 3 & -10 & 2 & -4 & -8 & 16 & [6,$2^2$] & 225 \\ \hline
        1 & 0 & 5 & -10 & -10 & 10 & 4 & [$4^2$,2] & 252 \\ \hline
        1 & -4 & -3 & 2 & 10 & 6 & -12 & [4,$2^4$] & 300 \\ \hline
        1 & -10 & \cellcolor{lightgray}{15} & 20 & -20 & -30 & \cellcolor{lightgray}{24} & [$2^5$] & 42 \\ \hline
    \end{tabular}
    \caption{Eigenvalues of $\mathcal{A}_{10}$}
    \label{tab_for_n=5}
\end{table}
\begin{table}[H]
    \centering
    \rotatebox{0}{\begin{tabular}{|l|l|l|l|l|l|l|l|l|l|l|l|l|}
    \hline
        [$1^6$] & [2,$1^4$] & [$2^2$,$1^2$] & [$2^3$] & [3,$1^3$] & [3,2,1] & [$3^2$] & [4,$1^2$] & [4,2] & [5,1] & [6] & Rep. & Dim. \\ \hline
        1 & 30 & 180 & 120 & 160 & 960 & 640 & 720 & 1440 & 2304 & 3840 & [12] & 1 \\ \hline
        1 & \cellcolor{lightgray}{19} & \cellcolor{lightgray}{48} & -12 & \cellcolor{lightgray}{72} & 80 & -64 & \cellcolor{lightgray}{192} & -144 & \cellcolor{lightgray}{192} & -384 & [10,2] & 54 \\ \hline
        1 & 12 & 27 & \cellcolor{lightgray}{30} & 16 & 24 & -8 & -18 & \cellcolor{lightgray}{108} & -144 & -48 & [8,4] & 275 \\ \hline
        1 & 9 & -12 & -12 & 22 & -60 & 16 & 12 & -24 & -48 & \cellcolor{lightgray}{96} & [8,$2^2$] & 616 \\ \hline
        1 & 9 & 33 & -27 & -8 & \cellcolor{lightgray}{120} & \cellcolor{lightgray}{136} & -78 & -114 & -48 & -24 & [$6^2$] & 132 \\ \hline
        1 & 4 & 3 & -2 & -8 & 0 & -24 & -18 & -4 & 32 & 16 & [6,4,2] & 2673 \\ \hline
        1 & 0 & -21 & 6 & 4 & 12 & 16 & -6 & 12 & 24 & -48 & [6,$2^3$] & 1925 \\ \hline
        1 & 0 & 15 & \cellcolor{lightgray}{30} & -20 & -60 & 40 & 30 & -60 & 24 & 0 & [$4^3$] & 462 \\ \hline
        1 & -3 & 3 & -9 & -8 & 0 & 4 & 24 & 24 & -24 & -12 & [$4^2$,$2^2$] & 2640 \\ \hline
        1 & -8 & 3 & 6 & 12 & 20 & -16 & -6 & -36 & -24 & 48 & [4,$2^4$] & 1485 \\ \hline
        1 & -15 & 45 & -15 & 40 & -120 & 40 & -90 & 90 & 144 & -120 & [$2^6$] & 132 \\ \hline
    \end{tabular}}
    \caption{Eigenvalues of $\mathcal{A}_{12}$}
    \label{tab_for_n=6}
\end{table}

\pagenumbering{gobble}

\begin{table}[H]
\small
    \centering
    \rotatebox{90}{\begin{tabular}{|l|l|l|l|l|l|l|l|l|l|l|l|l|l|l|l|l|}
    \hline
        [$1^7$] & [2, $1^5$] & [$2^2$,$1^3$] & [$2^3$,1] & [3,$1^4$] & [3,2,$1^2$] & [3,$2^2$] & [$3^2$,1] & [4,$1^3$] & [4,2,1] & [4,3] & [5,$1^2$] & [5,2] & [6,1] & [7] & Rep. & Dim. \\ \hline
        1 & 42 & 420 & 840 & 280 & 3360 & 3360 & 4480 & 1680 & 10080 & 13440 & 8064 & 16128 & 26880 & 46080 & [14] & 1 \\ \hline
        1 & \cellcolor{lightgray}{29} & \cellcolor{lightgray}{160} & 60 & \cellcolor{lightgray}{150} & \cellcolor{lightgray}{760} & -280 & 320 & \cellcolor{lightgray}{640} & \cellcolor{lightgray}{720} & -1120 & \cellcolor{lightgray}{1824} & -1344 & \cellcolor{lightgray}{1920} & -3840 & [12,2] & 77 \\ \hline
        1 & 20 & 79 & \cellcolor{lightgray}{114} & 60 & 148 & \cellcolor{lightgray}{368} & -184 & 118 & 180 & -112 & -120 & \cellcolor{lightgray}{816} & -1104 & -384 & [10,4] & 637 \\ \hline
        1 & 17 & 16 & -84 & 66 & -56 & -136 & -64 & 160 & -432 & 224 & 96 & -192 & -384 & \cellcolor{lightgray}{768} & [10,$2^2$] & 1365 \\ \hline
        1 & 15 & 69 & 39 & 10 & 228 & -42 & \cellcolor{lightgray}{376} & -102 & 90 & \cellcolor{lightgray}{588} & -360 & -504 & -264 & -144 & [8,6] & 1001 \\ \hline
        1 & 10 & 9 & 14 & 10 & -42 & -52 & -64 & -22 & 100 & -112 & -100 & -24 & 176 & 96 & [8,4,2] & 12012 \\ \hline
        1 & 6 & -39 & -18 & 22 & -78 & 84 & 112 & 6 & -36 & 48 & -36 & 72 & 144 & -288 & [8,$2^3$] & 7644 \\ \hline
        1 & 7 & 21 & -49 & -14 & 84 & 14 & 56 & -70 & -182 & -196 & 56 & 168 & 56 & 48 & [$6^2$,2] & 6435 \\ \hline
        1 & 4 & 15 & 50 & -20 & -60 & 80 & -40 & -10 & -140 & 80 & 104 & -144 & 80 & 0 & [6,$4^2$] & 9009 \\ \hline
        1 & 1 & -9 & -13 & -8 & 12 & -16 & -28 & -4 & 64 & 68 & 44 & 12 & -76 & -48 & [6,4,$2^2$] & 42042 \\ \hline
        1 & -4 & -29 & 42 & 12 & 52 & -16 & 32 & -14 & -36 & -112 & 24 & -48 & -96 & 192 & [6,$2^4$] & 14014 \\ \hline
        1 & -3 & 15 & 15 & -20 & -60 & -60 & 100 & 60 & 0 & -60 & -36 & 108 & -60 & 0 & [$4^3$,2] & 12012 \\ \hline
        1 & -7 & 7 & -21 & 0 & 28 & 56 & -28 & 28 & 0 & -28 & -84 & -84 & 84 & 48 & [$4^2$,$2^3$] & 21450 \\ \hline
        1 & -13 & 25 & 15 & 30 & -20 & -70 & -40 & -50 & -90 & 140 & 24 & 168 & 120 & -240 & [4,$2^5$] & 7007 \\ \hline
        1 & -21 & 105 & -105 & 70 & -420 & 210 & 280 & -210 & 630 & -420 & 504 & -504 & -840 & 720 & [$2^7$] & 429 \\ \hline
    \end{tabular}}
    \caption{Eigenvalues of $\mathcal{A}_{14}$}
    \label{tab_for_n=7}
\end{table}
\end{document}